\documentclass{article}

\usepackage{amsmath,amsthm,amssymb}
\usepackage[latin1]{inputenc}
\usepackage{subfigure}
\usepackage[usenames,dvipsnames,svgnames,table]{xcolor}
\usepackage{qtree}
\usepackage{youngtab}
\usepackage{young}
\usepackage{cite}
\usepackage{array}
\input xy
\xyoption{all}

\definecolor{light-gray}{gray}{0.85}
\definecolor{dark-gray}{gray}{0.25}

\usepackage{tikz}
\newcounter{x}
\newcounter{y}
\newcounter{z}

\newcommand\xaxis{210}
\newcommand\yaxis{-30}
\newcommand\zaxis{90}

\newcommand\topside[3]{
  \fill[fill=white, draw=black,shift={(\xaxis:#1)},shift={(\yaxis:#2)},
  shift={(\zaxis:#3)}] (0,0) -- (30:1) -- (0,1) --(150:1)--(0,0);
}

\newcommand\leftside[3]{
  \fill[fill=light-gray, draw=black,shift={(\xaxis:#1)},shift={(\yaxis:#2)},
  shift={(\zaxis:#3)}] (0,0) -- (0,-1) -- (210:1) --(150:1)--(0,0);
}

\newcommand\rightside[3]{
  \fill[fill=dark-gray, draw=black,shift={(\xaxis:#1)},shift={(\yaxis:#2)},
  shift={(\zaxis:#3)}] (0,0) -- (30:1) -- (-30:1) --(0,-1)--(0,0);
}

\newcommand\cube[3]{
  \topside{#1}{#2}{#3} \leftside{#1}{#2}{#3} \rightside{#1}{#2}{#3}
}

\newcommand\planepartition[1]{
 \setcounter{x}{-1}
  \foreach \a in {#1} {
    \addtocounter{x}{1}
    \setcounter{y}{-1}
    \foreach \b in \a {
      \addtocounter{y}{1}
      \setcounter{z}{-1}
      \foreach \c in {1,...,\b} {
        \addtocounter{z}{1}
        \cube{\value{x}}{\value{y}}{\value{z}}
      }
    }
  }
}

\input xy \xyoption{all}
\newcommand\qbin[3]{\left[\begin{matrix} #1 \\ #2 \end{matrix} \right]_{#3}}
\newcommand{\ten}{10}
\newcommand{\eleven}{11}
\newcommand{\twelve}{12}

\newcommand\Altbax[1]{\mathrm{AltBax}_{#1}}
\newcommand\Twin[1]{\mathrm{Twin}_{#1}}
\newcommand\bt[1]{\mathbb{BT}_{#1}}
\newcommand\cbt[1]{\mathbb{CBT}_{#1}}
\newcommand\code{\mathrm{code}}

\newtheorem{theorem}{Theorem}[section]

\newtheorem{corollary}[theorem]{Corollary}
\newtheorem{prop}[theorem]{Proposition}

\theoremstyle{definition}
\newtheorem{definition}[theorem]{Definition}
\newtheorem{example}[theorem]{Example}

\theoremstyle{remark}
\newtheorem{remark}[theorem]{Remark}

\begin{document}

\author{Kevin Dilks}

\title{Involutions on Baxter Objects}

\date{\today}

\maketitle

\begin{abstract}
Baxter numbers are known to count several families of combinatorial objects, all of which come equipped with natural involutions. In this paper, we add a combinatorial family to the list, and show that the known bijections between these objects respect these involutions. We also give a formula for the number of objects fixed under this involution, showing that it is an instance of Stembridge's ``$q=-1$ phenomenon''.
\end{abstract}


\section{Introduction}
The Baxter numbers are given by $B(n):=\sum_{k=0}^{n-1} \Theta_{k,n-k-1}$ where 
\begin{equation}
\label{baxternumber}
\Theta_{k,\ell}=\frac{\binom{n+1}{k}\binom{n+1}{k+1}\binom{n+1}{k+2}}{\binom{n+1}{1}\binom{n+1}{2}}
\end{equation}
 for $n=k+\ell+1$. The summand $\Theta_{k,\ell}$ counts many things, defined below, and illustrated in the Appendix: 
\noindent
\begin{description}
 \item[(A)] Baxter permutations in $\mathfrak{S}_{n}$ with $k$ ascents and $\ell$ descents.~\cite{NumBax}
 \item[(B)] Baxter permutations in $\mathfrak{S}_{n}$ with $k$ inverse ascents and $\ell$ inverse descents.
 \item[(C)] Twisted Baxter permutations in $\mathfrak{S}_{n}$ with $k$ inverse ascents and $\ell$ inverse descents.~\cite{LR}
 \item[(D)] Non-intersecting lattice paths from  \begin{align} & A_1=(0,2)\text{, } A_2=(1,1)\text{, and } A_3=(2,0) \text{ to} \label{latticepoints} \\ & B_1=(k,\ell +2), B_2=(k+1,\ell +1)\text{, } B_3=(k+2,\ell)\text{,} \nonumber \end{align} which we will call \emph{($k$,$\ell$)-Baxter paths}.~\cite{DG2}
 \item[(E)] Standard Young tableaux of shape $3\times n$ with no consecutive entries in any row, and $k$ instances of $(i,i+1)$ in the union of the first and third rows, which we will call \emph{ ($k$,$\ell$)-Baxter tableaux}.~\cite{DG1}
 \item[(F)] Diagonal rectangulations of size $n$, where $k$ is the number of times the interior of the diagonal is intersected vertically, and $\ell$ is the number of times it is intersected horizontally.~\cite{LR}
 \item[(G)] Plane partitions in a $k\times \ell\times 3$ box, which we will call \emph{Baxter plane partitions}.
\end{description}

Recall that a permutation $w=w_1\ldots w_n$ has a descent at position $i$ if $w_i>w_{i+1}$. A permutation $w$ has an inverse descent at position $i$ if $w^{-1}$ has a descent as position $i$, which is equivalent to $i+1$ appearing to the left of $i$ in $w$.

Baxter permutations are those that avoid the patterns 3-14-2 and 2-41-3, where an occurrence of the pattern 3-14-2 in a permutation $w=w_1\ldots w_n$ means there exists a quadruple of indices $\{i,j,j+1,k\}$ with $i<j<j+1<k$ and $w_j< w_k< w_i< w_{j+1}$ (and similarly for 2-14-3)\footnote{Such patterns are sometimes called {\textit{vincular} patterns}.}. For example, 25314 contains an instance of the patten 2413, but not 2-41-3. For $n=4$, there are $B(4)= 22$ Baxter permutations in $\mathfrak{S}_4$, with the only excluded ones being 2413 and 3142. Twisted Baxter permutation have a syntactically similar definition, being those that avoid 2-41-3 and 3-41-2. Call these larger sets counted by $B(n)$ a set of \emph{Baxter objects of order $n$}, and their subsets counted by $\Theta_{k,\ell}$ a set of \emph{Baxter objects of order ($k$,$\ell$)}. Each of these subsets has a natural involution that preserves $k$ and $\ell$:

\begin{itemize}
 \item Conjugation by the longest permutation $w_0$ for (A), (B), and (C).
 \item Rotation by $180^{\circ}$ about a central point for (D) and (F)
 \item Sch\"utzenberger evacuation for (E), which in the special case of a rectangular tableaux with $N$ boxes corresponds to rotating the tableaux $180^{\circ}$ and then replacing every label $i$ with $N+1-i$.
 \item Taking the complement of a plane partition in the $k\times l\times 3$ box for (G).
\end{itemize}

Since Baxter permutations are closed under taking inverses~\cite{LR}, the map $w\mapsto w^{-1}$ provides an obvious bijection between Baxter objects (A) and (B). There are known bijections due to Dulucq and Guibert between the Baxter objects (A), (D) and (E) (see~\cite{DG1},\cite{DG2}), and also between the objects (B), (C) and (F) due to Law and Reading (see~\cite{LR}). We will also show the equivalence of objects (D) and (G). Section 2 is devoted to the proof of the following theorem:

\begin{theorem}
\label{Theorem 1}
 The given bijections between the above 7 classes of Baxter objects of order ($k$,$\ell$), commute with their respective involutions.
\end{theorem}

Since the bijections commute with the respective involutions, this means the number of Baxter objects of order ($k$,$\ell$) fixed under involution is the same for all 7 classes of Baxter objects. Denote this common number $\Theta_{k,\ell}^{\circlearrowleft}$, and introduce a q-analogue of $\Theta_{k,\ell}$, 

\begin{equation}
\label{defofqanalog}
\Theta_{k,\ell}(q) := \frac{\qbin{n+1}{k}{q}\qbin{n+1}{k+1}{q}\qbin{n+1}{k+2}{q}}{\qbin{n+1}{1}{q}\qbin{n+1}{2}{q}}
\end{equation} 

\noindent where $n=k+\ell +1$, $\qbin{n}{i}{q}=\frac{[n]!_q}{[k]!_q[n-k]!_q}$, $[m]!_q=[m]_q[m-1]_q\ldots[1]_q$, and $[j]_q=1+q+\ldots+q^{j-1}$.

\begin{theorem}
\label{Theorem 2}
 $\Theta_{k,\ell}(q)$ lies in $\mathbb{N}[q]$, has symmetric coefficients, and satisfies $[\Theta_{k,\ell}(q)]_{q=-1}=\Theta_{k,\ell}^{\circlearrowleft}$.
\end{theorem}

The proof of Theorem \ref{Theorem 2} is given in Section \ref{Section 3}, using a result of Stembridge from the theory of plane partitions.

\section{Proof of Theorem~\ref{Theorem 1}}
\label{Section 2}

\subsection{Objects (D) and (G)}

A plane partition is an array $(\pi_{i,j})_{i,j\geq 1}$ of non-negative integers with finitely many non-zero entries that weakly decrease along rows and columns. The plane partitions inside an $a\times b\times c$ box are those where $\pi_{i,j}\leq c$, and $\pi_{i,j}=0$ if $i>a$ or $j>b$. Its complement in the $a\times b\times c$ box is the plane partition given by $\pi'_{i,j}=c-\pi_{a-i,b-j}$ for $1\leq i\leq a$ and $1\leq j\leq b$ and 0 elsewhere.

\begin{figure}[h]
\label{triptopp}

\begin{tikzpicture}[scale=.5]
 \draw[gray,very thin] (0,0) grid (6,6);
 \draw[very thick] (2,0) node[anchor=north] {(2,0)} -- (5,0) -- (5,2) -- (6,2) -- (6,4) node[anchor=west] {(6,4)};
 \draw[very thick] (1,1) node[anchor=north east] {(1,1)} -- (3,1) -- (3,3) -- (4,3) -- (4,4) -- (5,4) -- (5,5) node[anchor=south west] {(5,5)};
 \draw[very thick] (0,2) node[anchor=east] {(0,2)} -- (1,2) -- (1,4) -- (3,4) -- (3,6) -- (4,6) node[anchor=south] {(4,6)};
 \filldraw (0,2) circle (2pt)
           (1,1) circle (2pt)
           (2,0) circle (2pt)
           (4,6) circle (2pt)
           (5,5) circle (2pt)
           (6,4) circle (2pt);

 \begin{scope}[xshift=200mm,yshift=50mm]
  \planepartition{{3,3,3,2},{3,3,3,1},{3,2,1},{3,2,1}}
 \end{scope}
 
 \begin{scope}[xshift=120mm,yshift=80mm]
 \planepartition{{1,1,1},{1,1,1},{1},{1}}
 \end{scope}
 
   \begin{scope}[xshift=120mm,yshift=40mm]
 \planepartition{{1,1,1,1},{1,1,1},{1,1},{1,1}}
 \end{scope}
 
  \begin{scope}[xshift=120mm,yshift=0mm]
 \planepartition{{1,1,1,1},{1,1,1,1},{1,1,1},{1,1,1}}
 \end{scope}
 
\end{tikzpicture}
\caption{Example of the map from triples of non-intersecting lattice paths as in \eqref{latticepoints} to plane partitions in a $k\times\ell\times 3$ box, for $k=\ell=4$. (Tikz code courtesy of Jang Soo Kim)}
\end{figure}
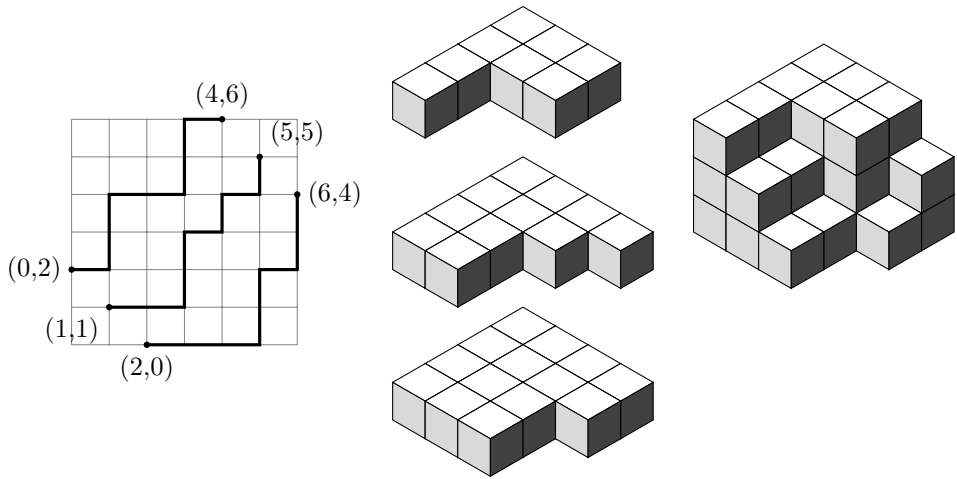

\begin{theorem}
 There is a bijection between $(k,\ell)$-Baxter paths and plane partitions in a $k \times \ell \times 3$ box, which equivariantly takes conjugation by $w_0$ to complementation of a plane partition.\footnote{Thanks to Jang Soo Kim for noting this connection.}
\end{theorem}

\begin{proof}
Each individual lattice path from $A_i$ to $B_i$ naturally corresponds to a partition $\lambda_i$ inside of a $k \times \ell$ box, (our convention will be to take $\lambda_i$ to be the part of the $k \times \ell$ box with $A_i$ and $B_i$ as corners that lies above the given lattice path). The non-intersecting condition is equivalent to requiring $\lambda_3\subseteq\lambda_2\subseteq\lambda_1$, which is precisely the condition necessary for a triple of partitions to form the layers of a plane partition when stacked. Additionally, one can see the involution on lattice paths (which is $180^{\circ}$ rotation) corresponds to taking $\lambda_3\subseteq\lambda_2\subseteq\lambda_1$ to $\lambda_1^c\subseteq\lambda_2^c\subseteq\lambda_3^c$, where $\lambda^c$ is the complement of $\lambda$ in the $k \times \ell$ box, which is the same as taking the complement of the plane partition in the $k \times \ell \times 3$ box.
\end{proof}

\newpage
\subsection{Objects (A), (D), and (E)}

One fundamental intermediate object in bijections between Baxter objects is a special sub-class of pairs of binary trees.

A \emph{binary tree} is a rooted plane tree where every node has at most two children, denoted the left and right child. A \emph{complete binary tree} is a binary tree where every node is either a leaf, or has has exactly two children. Let $\bt{n}$ denote the set of binary trees with $n$ nodes, and $\cbt{2n+1}$ the set of complete binary trees on $2n+1$ nodes.

If we truncate all of the leaves from a complete binary tree on $2n+1$ nodes, we're left with a binary tree on $n$ nodes. If we have a binary tree on $n$ nodes, we can extend it to a binary tree on $2n+1$ nodes by adding leaves to every node with 0 or 1 children. The processes of truncation and extension can clearly be seen to be inverse to each other.

\begin{definition}
Let $\mathrm{Trunc}: \cbt{2n+1}\mapsto\bt{n}$ be the bijection from complete binary trees on $2n+1$ nodes to binary trees on $n$ nodes obtained by truncating leaves, with inverse map called $\mathrm{Extend}$. Let $\mathrm{Trunc}^{\times 2}: \cbt{2n+1}\times\cbt{2n+1}\mapsto\bt{n}\times\bt{n}$ (resp. $\mathrm{Extend}^{\times 2}$) be the corresponding maps on pairs of trees. 
\end{definition}

We call a leaf a left (resp. right) leaf if it is a left (resp. right) child of its parent.

\begin{definition}
Let $\code:\cbt{2n+1}\mapsto \{0,1\}^{n-1}$ be the function that reads off the pattern of left and right leaves in a complete binary tree from left to right (excluding the left-most left leaf and right-most right leaf) by assigning a 1 to left leaves and a 0 to right leaves
\end{definition}

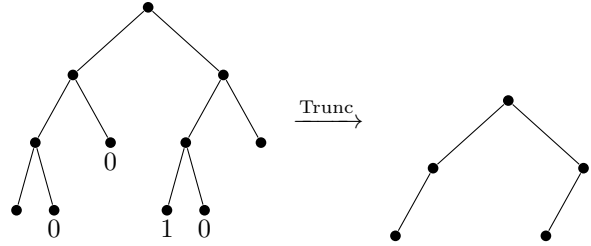
\begin{figure}[!h]
\centering
$\begin{array}{ccc}
\begin{tikzpicture} [inner sep=.5mm, level distance=9mm, level 4/.style={sibling distance=5mm}, level 3/.style={sibling distance=5mm}, level 2/.style={sibling distance=10mm}, level 1/.style={sibling distance=20mm}] \node[fill,circle] {} child { node[fill,circle]{} child { node[fill,circle]{} child {node[fill,circle]{}} child{node[fill,circle,label=below:0]{}}} child {node[fill,circle,label=below:0]{}} } child { node[fill,circle]{} child { node[fill,circle]{} child{node[fill,circle,label=below:1]{}} child{node[fill,circle,label=below:0]{}} } child{node[fill,circle]{}} }; \end{tikzpicture} & \raisebox{15mm}{$\xrightarrow{\mathrm{Trunc}}$} & \begin{tikzpicture} [inner sep=.5mm, level distance=9mm, level 4/.style={sibling distance=5mm}, level 3/.style={sibling distance=5mm}, level 2/.style={sibling distance=10mm}, level 1/.style={sibling distance=20mm}] \node[fill,circle] {} child { node[fill,circle]{} child{node[fill,circle]{}} child[missing]} child {node[fill,circle]{} child{node[fill,circle]{}} child[missing]}; \end{tikzpicture}
\end{array}$
\caption{Example of \textrm{Trunc} taking an element of $\cbt{11}$ with leaf code 0010 to an element of $\bt{5}$.}
\end{figure}

We will mainly be interested in pairs of complete binary trees satisfying a compatibility relation.

\begin{definition}
Let $\Twin{n}\subset\cbt{2n+1}\times\cbt{2n+1}$ be the set of pairs of complete binary trees, $\left(T_L,T_R\right)$, where $\code(T_L)$ is the same as $\code(T_R)$ if we interchange $0$'s and $1$'s. Let $\widetilde{\Twin{n}}$ be the image of $\Twin{n}$ under $\mathrm{Trunc}^{\times 2}$.
\end{definition}

\begin{figure}[!hbp]
\centering
$\begin{array}{c}
\raisebox{\height}{$\Bigg($}
\begin{tikzpicture}
[inner sep=.5mm, level distance=8mm, level 3/.style={sibling distance=4mm}, level 2/.style={sibling distance=7mm}] 
    \node[fill,circle] {}  	
      child{ node[fill,circle] (b) {}
	child{ node[fill,circle] (c) {}
	  child { node[fill,circle] (a) {} }
	  child { node[fill,circle] (alpha) {} }
	}
	child { node[fill,circle] (d) {} child { node[fill,circle] (beta) {}} child { node[fill,circle] (gamma) {}}}
      }
      child { node[fill,circle] (e) {} child {node[fill,circle] (delta) {}} child { node[fill,circle] (f) {} }
      };

\draw (1.75,-2) node {,};
\begin{scope}[xshift=35mm]
[inner sep=.5mm, level distance=8mm, level 3/.style={sibling distance=4mm}, level 2/.style={sibling distance=4mm}] 
    \node[fill,circle] {}  	
      child { node[fill,circle] (f) {}
	child { node[fill,circle] (g) {} }
	child{ node[fill,circle] (h) {} child { node[fill,circle] (Alpha) {}} child { node[fill,circle] (Beta) {}}}
      }
      child{ node[fill,circle] (i) {}
	child{ node[fill,circle] (j) {} child { node[fill,circle] (Gamma) {}} child { node[fill,circle] (Delta) {}}  }
	child { node[fill,circle] (k) {} }
	};
\end{scope}
\end{tikzpicture}
\raisebox{\height}{$\Bigg)$}\\
\downarrow \\
\raisebox{\height}{$\Bigg($}
\begin{tikzpicture}
[inner sep=.5mm, level distance=9mm, level 3/.style={sibling distance=5mm}, level 2/.style={sibling distance=10mm}] 
    \node[fill,circle] {}  	
      child {node[fill,circle] {}
	child {node[fill,circle] {}}
	child {node[fill,circle] {}}
      }
      child {node[fill,circle] {}
      };
\draw (1.25,-1.5) node {,};
\begin{scope}[xshift=25mm]
[inner sep=.5mm, level distance=9mm, level 3/.style={sibling distance=5mm}, level 2/.style={sibling distance=10mm},xshift=100cm] 
    \node[fill,circle] {}  	
      child {node[fill,circle] {}
	child[missing] {node[fill,circle] {m1}}
	child {node[fill,circle] {}}
      }
      child{node[fill,circle] {}
	child{node[fill,circle] {}}
	child[missing] {node[fill,circle] {m2}}
	};
\end{scope}
\end{tikzpicture}
\raisebox{\height}{$\Bigg)$}\\
\end{array}$

\caption{Example of map from $\Twin{n}$ to $\widetilde{\Twin{n}}$ for $n=5$.}
\label{twintrees1}
\end{figure}
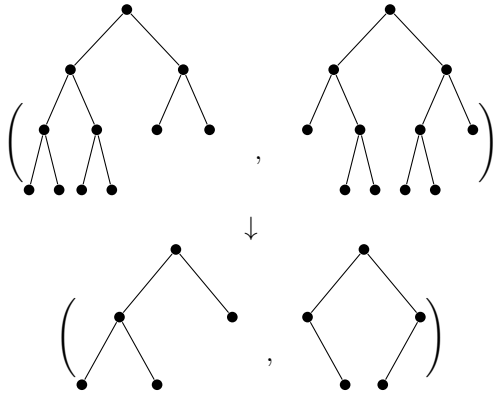

\noindent Clearly $\Twin{n}$ and $\widetilde{\Twin{n}}$ are in bijection.

\begin{definition}[{\cite[\S 1.2]{Stanley2}}]
Given a word $w=w_1w_2\ldots w_n$ with distinct letters in $\mathbb{N}_{\geq 1}$, recursively define a binary tree $\mathrm{incr}(w)$ called the \emph{increasing binary tree} for $w$ by saying that if $w=uxv$ with $x=\mathrm{min}\{w_1,\ldots w_n\}$, then $\mathrm{incr}(w)$ has $x$ as its root, $\mathrm{incr}(u)$ as its left subtree, and $\mathrm{incr}(v)$ as its right subtree. Similarly, recursively define a binary tree $\mathrm{decr}(w)$ called the \emph{decreasing binary tree} of $w$ by saying that if $w=uxv$ with $x=\mathrm{max}\{w_1,\ldots w_n\}$, then $\mathrm{decr}(w)$ has $x$ as its root, $\mathrm{decr}(u)$ as its left subtree, and $\mathrm{incr}(v)$ as its right subtree.
\end{definition}

While this process gives a labelled binary tree, we will only consider $\mathrm{incr}(w)$ and $\mathrm{decr}(w)$ to be the underlying unlabelled binary tree.

\begin{definition}
Let $\Psi:S_n\mapsto \bt{n}\times\bt{n}$ be the map that sends a permutation $w$ to the pair of binary trees $\left(\mathrm{incr}(w),\mathrm{decr}(w)\right)$.
\end{definition}

\begin{theorem}[Dulucq and Guibert, \cite{DG1}]
$\Psi: \mathrm{Bax}_n \mapsto \widetilde{\mathrm{Twin}_n}$ is a bijection.
\end{theorem}

It is known that if $w$ has $k$ ascents and $\ell$ descents, then $\mathrm{Extend}(\mathrm{incr}(w))$ will have $k+1$ left leaves and $\ell+1$ right leaves.

\begin{definition}

Say $w\in\ S_n$ is alternating if $w_1<w_2>w_3<w_4\ldots\phantom{A}$. Let $\mathrm{AltBax_{2n}}$ denote the set of alternating Baxter permutations of length $2n$.
\end{definition}

Recall that alternating permutations have the property that $\mathrm{incr}(w)$ (resp. $\mathrm{decr}(w)$) is a complete binary tree if we add a left-most left leaf (resp. right-most right leaf)~\cite[Prop. 1.3.14]{Stanley2}.

\begin{corollary}
The function $\Psi$ is a bijection from alternating Baxter permutations of length $2n$ to all pairs of complete binary trees with $2n+1$ nodes each.
\label{altbaxtocbt}
\end{corollary}

\begin{definition}
\label{treeinvdef}
The natural involution on pairs of complete binary trees (which has $\mathrm{Twin}_n$ as a subset) is taking the mirror reflection of each tree, and then swapping the two trees (see Figure~\ref{twintrees2}). 
\end{definition}

\begin{prop}
$\Psi$ equivariantly maps permutations with the action of conjugation by the longest element to pairs of twin trees with this involution action.
\end{prop}
This proposition is obvious from the definition of $\Psi$ in terms of increasing/decreasing trees.

\begin{figure}[!hbp]

\centering
$\begin{array}{cc}
\begin{array}{c}
 43512 \\
\updownarrow \\
\raisebox{\height}{$\Bigg($}
\begin{tikzpicture}
[inner sep=0mm, level distance=9mm, level 3/.style={sibling distance=5mm}, level 2/.style={sibling distance=10mm}] 
    \node[circle,draw] {1}  	
      child {node[circle,draw] {3}
	child {node[circle,draw] {4}}
	child {node[circle,draw] {5}}
      }
      child {node[circle,draw] {2}
      };
\draw (1.25,-1.5) node {,};
\begin{scope}[xshift=25mm]
[inner sep=0mm, level distance=9mm, level 3/.style={sibling distance=5mm}, level 2/.style={sibling distance=10mm},xshift=100cm] 
    \node[circle,draw] {5}  	
      child {node[circle,draw] {4}
	child[missing] {node {m1}}
	child {node[circle,draw] {3}}
      }
      child{node[circle,draw] {2}
	child{node[circle,draw] {1}}
	child[missing] {node {m2}}
	};
\end{scope}
\end{tikzpicture}
\raisebox{\height}{$\Bigg)$}\\
\downarrow \\
\raisebox{\height}{$\Bigg($}
\begin{tikzpicture}
[inner sep=.5mm, level distance=8mm, level 3/.style={sibling distance=4mm}, level 2/.style={sibling distance=7mm}] 
    \node[fill,circle] {}  	
      child{ node[fill,circle] {}
	child{ node[fill,circle] {}
	  child{ [dashed] node[fill,circle] {} }
	  child {[dashed] node[fill,circle,label=below:1] {} }
	}
	child { node[fill,circle] {} child {[dashed] node[fill,circle,label=below:0] {}} child {[dashed] node[fill,circle,label=below:1] {}}}
      }
      child { node[fill,circle] {} child {[dashed] node[fill,circle,label=below:0] {}} child{ [dashed] node[fill,circle] {} }
      };

\draw (1.75,-2) node {,};
\begin{scope}[xshift=35mm]
[inner sep=0mm, level distance=8mm, level 3/.style={sibling distance=4mm}, level 2/.style={sibling distance=4mm}] 
    \node[fill,circle] {}  	
      child { node[fill,circle] {}
	child{ [dashed] node[fill,circle] {} }
	child{ node[fill,circle] {} child {[dashed] node[fill,circle,label=below:0] {} } child {[dashed] node[fill,circle,label=below:1] {}}}
      }
      child{ node[fill,circle] {}
	child{ node[fill,circle] {} child {[dashed] node[fill,circle,label=below:0] {}} child {[dashed] node[fill,circle,label=below:1] {}}  }
	child{ [dashed] node[fill,circle] {} }
	};
\end{scope}
\end{tikzpicture}
\raisebox{\height}{$\Bigg)$}\\
\end{array}

&

\begin{array}{c}
 35142 \\
\updownarrow \\
\raisebox{\height}{$\Bigg($}
\begin{tikzpicture}
[inner sep=.0mm, level distance=9mm, level 3/.style={sibling distance=5mm}, level 2/.style={sibling distance=10mm},xshift=100cm] 
    \node[circle,draw] {1}  	
      child {node[circle,draw] {3}
	child[missing] {node {m1}}
	child {node[circle,draw] {5}}
      }
      child{node[circle,draw] {2}
	child{node[circle,draw] {4}}
	child[missing] {node {m2}}
	};
	\draw (1.25,-1.5) node {,};	
\begin{scope}[xshift=25mm]
[inner sep=0mm, level distance=9mm, level 3/.style={sibling distance=5mm}, level 2/.style={sibling distance=10mm}] 
    \node[circle,draw] {5}  	
      child {node[circle,draw] {3}}
			child {node[circle,draw] {4}
				child {node[circle,draw] {1}}
      	child {node[circle,draw] {2}}
      };
\end{scope}
\end{tikzpicture}
\raisebox{\height}{$\Bigg)$}\\
\downarrow \\
\raisebox{\height}{$\Bigg($}
\begin{tikzpicture}
[inner sep=.5mm, level distance=8mm, level 3/.style={sibling distance=4mm}, level 2/.style={sibling distance=7mm}]
    \node[fill,circle] {}  	
      child { node[fill,circle] {}
	child{ [dashed] node[fill,circle] {} }
	child{ node[fill,circle] {} child {[dashed] node[fill,circle,label=below:0] {}} child {[dashed] node[fill,circle,label=below:1] {}}}
      }
      child{ node[fill,circle] {}
	child{ node[fill,circle] {} child {[dashed] node[fill,circle,label=below:0] {}} child {[dashed] node[fill,circle,label=below:1] {}}  }
	child{ [dashed] node[fill,circle] {} }
	};

\draw (1.75,-2) node {,};
\begin{scope}[xshift=35mm]
[inner sep=0mm, level distance=8mm, level 3/.style={sibling distance=4mm}, level 2/.style={sibling distance=4mm}] 
 
    \node[fill,circle] {}  
    	child { node[fill,circle] {} child[dashed] child {[dashed] node[fill,circle,label=below:1] {}} 
      }	
      child{ node[fill,circle] {}
	child{ node[fill,circle] {}
	  child {[dashed] node[fill,circle,label=below:0] {} }
	  child {[dashed] node[fill,circle,label=below:1] {} }
	}
	child { node[fill,circle] {} child {[dashed] node[fill,circle,label=below:0] {} } child {[dashed] node[fill,circle] {} }}
      };
\end{scope}
\end{tikzpicture}
\raisebox{\height}{$\Bigg)$}\\
\end{array}\\
\end{array}$

\caption{Map from $\mathrm{Bax}_n$ to $\Twin{n}$ for $n=5$, and the corresponding action under involution.}
\label{twintrees2}
\end{figure}
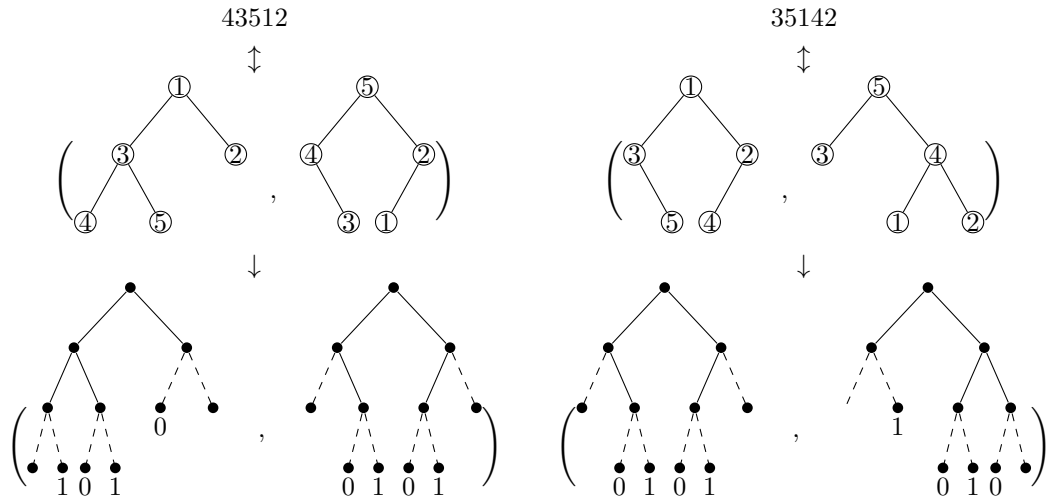

The equivalence of Baxter permutations fixed under conjugation by $w_0$ and triples of non-intersecting of lattice paths fixed under rotation, along with a number of other Baxter objects fixed under their respective involutions, is given by Felsner, Fusy, Noy, and Orden~\cite{BBFRO}. They follow from the fact that the bijections between the corresponding Baxter objects are all equivariant with respect to the natural involutions. 

\begin{theorem}[Dulucq and Guibert,~\cite{DG1}]
There is a bijection between elements of $\mathrm{Twin}_n$ with $k+1$ left leaves and $\ell+1$ right leaves, and ($k$,$\ell$)-Baxter paths..
\end{theorem}

\begin{prop}
The above bijection equivariantly takes the natural involution on $\mathrm{Twin}_n$ to rotation by $180^{\circ}$ on triples of non-intersecting lattice paths.
\end{prop}

\begin{proof}
Given a pair of twin trees $\left(T_L,T_R\right)$, the first path (resp. third path) arises from reading the internal nodes of $T_L$ (resp. $T_R$) in infix order, recording whether they are left or right children of their parents. The middle path is determined by $\mathrm{code}(T_L)$, which by the twin condition encodes the same information as $\mathrm{code}(T_R)$. 
\end{proof}

Felsner, Fusy, Noy, and Orden~\cite{BBFRO} have additionally shown the bijections to a number of other Baxter objects are also equivariant. One interesting Baxter family not included are Baxter tableaux, or $3\times n$ standard Young tableaux with no consecutive entries in the same row, which we will now look at.

Cori, Dulucq, and Viennot~\cite{CDV} begin by working with a larger set of objects, counted not by $B(n)$, but by $c_n^2$, where $c_n=\frac{1}{n+1}\binom{2n}{n}$ is the \emph{Catalan number}.

\begin{definition}
Let $Y_n^{(k)}$ be the language of all words in $\{1,2,\ldots,k\}$ such that each letter appears exactly $n$ times, and for any prefix of the word, $i$ appears at least as often as $i+1$. These are exactly the Yamanouchi words for standard Young tableaux of a $k\times n$ box\footnote{They are also referred to as \emph{stack words}, as they encode the permutations that can be sorted with $k-1$ stacks~\cite{Gire}.}. 
\end{definition}

A Yamanouchi word for a standard Young tableau is a word where the $i^{th}$ letter indicates which row of the tableau $i$ appears in.

Additionally, if we think of 1's as being $z$'s (corresponding to left parentheses) and 2's as being $\overline{z}$'s (corresponding to right parentheses), then $P_{z,\overline{z}}$, the language of well-formed parenthesis systems on the letters $\{z,\bar{z}\}$, corresponds to $\cup_{n\geq 0} Y_{n}^{(2)}$.

Evacuation can be defined more generally, but in the special case of standard Young tableaux of rectangular shape, it takes a particularly nice form.

\begin{definition}
Given a standard Young tableau $T$ of square shape with $N$ boxes, let $\mathrm{evac}(T)$ be the Young tableau we get by rotating $T$ by $180^\circ$, and then replacing each label $i$ with $N+1-i$.
\end{definition}

\noindent See Tables 1-4 at the end for examples. 

\newpage

This action also takes a nice form on the corresponding Yamanouchi words.

\begin{definition}
If $x=x_1 x_2\ldots x_{3n-1} x_{3n}\in Y_n^{(3)}$, then let $$\mathrm{evac}(x)=(3-x_{3n})(3-x_{3n-1})\ldots(3-x_2)(3-x_1).$$
\end{definition} 

\begin{example}
$\mathrm{evac}(112323)=121233$
\end{example}

We introduce an intermediate object, consisting of certain shuffles of two parenthesization systems.

\begin{definition}
Let $\mathrm{Shuffle}_{2n}$ be the set of all shuffles of $P_{a,\bar{a}}$ and $P_{b,\bar{b}}$ of length $2n$ such that for every prefix ending in $b$, the number of $a$'s is strictly greater than the number of $\bar{a}$'s.
\end{definition}

\begin{example}

$ab\bar{b}b\bar{a}a\bar{a}a\bar{b}\bar{a}\in\mathrm{Shuffle}_{10}$

\end{example}

\begin{theorem}[Cori, Dulucq, Viennot~\cite{CDV}]
There are bijections between $\mathrm{AltBax}_{2n}$, $\mathrm{Shuffle}_{2n}$, and pairs of complete binary trees with $2n+1$ nodes each.\\

\noindent In particular, each set has $c_n^2$ objects.
\label{chartproof}
\end{theorem}

The bijection $\beta:\mathrm{AltBax}_{2n}\to\mathrm{Shuffle}_{2n}$ will later be recalled in Definition~\ref{def12}.

Later, Dulucq and Guibert showed that there was an additional bijection to a special class of Yamanouchi words.

\begin{definition}
Let $Y_n^{(3)}(22)$ be the subset of $Y_n^{(3)}$ consisting of Yamanouchi words avoiding the consecutive pattern 22 (corresponding to Young tableaux with no consecutive entries in the middle row). Let $Y_n^{(3)}(11,22,33)$ be the subset of $Y_n^{(2)}$ consisting of Yamanouchi words avoiding the consecutive patterns 11, 22, or 33 (corresponding to Baxter tableaux).
\end{definition}

\begin{theorem}[Dulucq and Guibert~\cite{DG1}]
There is a bijection between $\mathrm{Shuffle}_{2n}$ and $Y_n^{(3)}(22)$. This bijection is given by the map that sends $\alpha=\alpha_1\alpha_2\ldots\alpha_{2n}\in\mathrm{Shuffle}_{2n}$ to $f(\alpha)=\Phi(\alpha_1)\Phi(\alpha_2)\ldots\Phi(\alpha_{2n})$, where

\begin{equation*}
\begin{cases}
\Phi(a)=1\\
\Phi(b)=21\\
\Phi(\bar{a})=23\\
\Phi(\bar{b})=3\\
\end{cases}
\end{equation*}

\end{theorem}

\begin{example}

\[\arraycolsep=1.7pt
\begin{array}{ccccccccccccl}
 & \Phi( & a & b & \bar{b} & b & \bar{a} & a & \bar{a} & a & \bar{b} & \bar{a} & )\\
 & = & 1 & 21 & 3 & 21 & 23 & 1 & 23 & 1 & 3 & 23 & \cr
\end{array}
\]



\end{example}

It is not immediately clear that all of the maps are necessarily equivariant with respect to their natural involutions. We will show that the original bijections of Cori, Dulucq, and Viennot on the objects counted by $c_n^2$ ($\mathrm{AltBax}_{2n}$, $\mathrm{Shuffle}_{2n}$, $\cbt{2n+1}\times\cbt{2n+1}$, $Y_n^{(3)}(22)$)  are equivariant with respect to their involutions. Then an equivariant bijection from the Baxter tableaux (or equivalently, Yamanouchi words in $Y_{n}^{(3)}(11,22,33)$) to $\mathrm{Twin}_n$ is obtained by restricting the equivariant bijection from Yamanouchi words in $Y_{n}^{(3)}(22)$ to all pairs of complete binary trees. 

\begin{figure}[!h]
\centering
\xymatrix @R=1pc @C=2pc{ & \cbt{n}\times\cbt{n} \ar@{<->}[r] ^{\Psi} & \mathrm{AltBax}_{2n} \ar@{<->}[r] ^\beta  & \mathrm{Shuffle}_{2n} \ar@{<->}[r] ^f & Y_n^{(3)}(22)  \\
           \mathrm{Bax}_n \ar@{<->}[r] ^{\mathrm{Extend}^{\times 2}\circ\Psi} & \mathrm{Twin}_n \ar@{<->}[rrr] \ar@{^{(}->}[u] & & & Y_n^{(3)}(11,22,33) \ar@{^{(}->}[u]   }
\caption{Diagram indicating the maps between objects counted by $c_n^2$, and their restriction to Baxter objects.}
\end{figure}

Lastly, we show that the bijection between the $\mathrm{Twin}_n$ and $\mathrm{Bax}_n$ is equivariant, making the composite map from Baxter tableaux to Baxter permutations equivariant.

First, we note that it is trivial to check that the map from Baxter permutations to alternating Baxter permutations of length $2n$ equivariantly takes conjugation by $w_0$ on $S_n$ to conjugation by $w_0$ on $S_{2n}$. The map $\Psi$ sends a Baxter permutation $w$ of length $n$ to a pair of twin trees in $\Twin{n}$, $\left(\mathrm{incr}(w),\mathrm{decr}(w)\right)$, equivariantly. But by Corollary~\ref{altbaxtocbt}, $\Psi^{-1}$ will equivariantly map this to $\Altbax{2n}$. So it suffices to check that the map from $\Altbax{2n}$ to $Y_{n}^{(3)}(22)$ is equivariant.

\begin{prop}
An equivalent formulation for the Baxter condition on permutations of length $n$ says that for every $p\in [n-1]$, we can either write the permutation as $$\pi=\pi' p \pi_{-} \pi_{+} (p+1) \pi''\quad \text{or}\quad \pi=\pi' (p+1) \pi_{+} \pi_{-}  p \pi'',$$ where the (possibly empty) subsequence $\pi_{-}$ (resp. $\pi_{+}$) consists of values less than $p$ (resp. greater than $p+1$).
\end{prop}

\begin{example}
For $\pi=5671342\in\mathrm{Bax}_7$ and $p=4$, we have $\pi'=\emptyset$, $\pi_{+}=67$, $\pi_{-}=13$, and $\pi''=2$.
\end{example}

The proof of this is straighforward, and left to the reader.

This allows us to construct a map from $\mathrm{AltBax}_{2n}$ to $\mathrm{Shuffle}_{2n}$, which is in fact the bijection referred to in Theorem~\ref{chartproof}.

\begin{definition}
\label{def12}
Let $\beta:\mathrm{AltBax}_{2n}\mapsto\mathrm{Shuffle}_{2n}$ be the map defined as follows:

Given a $\pi\in\mathrm{AltBax}_{2n}$, for each $p\in [2n-1]$, look at the relative order of $p$ and $p+1$, whether $\pi_{-}$ is empty, and whether $\pi_{+}$ is empty. We call this triple of information the $\emph{type}$ of $p$ (with respect to $\pi$). If $\beta(\pi)=\alpha$, for each of these 8 types, there are two possible strings of length two that $\alpha_p\alpha_{p+1}$ could be, listed in the figure below. Starting with $\alpha_1=a$, we can resursively construct $\alpha$ by noting that only one of the two choices for $\alpha_p\alpha_{p+1}$ will be consistent with what we already know $\alpha_p$ must be.
\end{definition}

\begin{figure}[!h]
\label{chart}
\centering
\begin{tabular}{|c|cccccc|c|c|}
\hline
Type & \multicolumn{6}{c|}{$\pi$} & $\alpha_p\alpha_{p+1}$  \\
\hline
\hline
Type 1 & $\pi'$ & $p$ & $\emptyset$ & $\emptyset$ & $p+1$ & $\pi''$ & $a\bar{a}$ or $b\bar{a}$ \\
\hline

Type 2 & $\pi'$ & $p$ & $\emptyset$ & $\pi_{+}$ & $p+1$ & $\pi''$ & $ab$ or $bb$ \\
\hline

Type 3 & $\pi'$ & $p$ & $\pi_{-}$ & $\emptyset$ & $p+1$ & $\pi''$ & $\bar{a}\bar{a}$ or $\bar{b}\bar{a}$\\
\hline

Type 4 & $\pi'$ & $p$ & $\pi_{-}$ & $\pi_{+}$ & $p+1$ & $\pi''$ & $\bar{a}b$ or $\bar{b}b$ \\
\hline

Type 5 & $\pi'$ & $p+1$ & $\emptyset$ & $\emptyset$ & $p$ & $\pi''$ & $a\bar{b}$ or $b\bar{b}$ \\
\hline

Type 6 & $\pi'$ & $p+1$ & $\emptyset$ & $\pi_{-}$ & $p$ & $\pi''$ & $\bar{a}\bar{b}$ or $\bar{b}\bar{b}$\\
\hline

Type 7 & $\pi'$ & $p+1$ & $\pi_{+}$ & $\emptyset$ & $p$ & $\pi''$ & $aa$ or $ba$ \\
\hline

Type 8 & $\pi'$ & $p+1$ & $\pi_{+}$ & $\pi_{-}$ & $p$ & $\pi''$ & $aa$ or $ba$ \\
\hline

\end{tabular}

\caption{Relations between $\pi\in\Altbax{2n}$, $\alpha\in\mathrm{Shuffle}_{2n}$, $\mathrm{incr}(\pi)$, and $\mathrm{decr}(\pi)$.}
\end{figure}

\begin{example}
As a working example, we will start with $\pi=2314\in\mathrm{AltBax}_{4}$. 

For $p=1$, we see that $p+1$ occurs before $p$, $\pi_{+}$ is non-empty, and $\pi_{-}$ is empty. This means $p=1$ is type 7, and so $\alpha_1\alpha_2$ is either $aa$ or $ba$. But since a shuffle word has to start with $a$, we know $\alpha_1\alpha_2=aa$.

For $p=2$, we see that $p$ occurs before $p+1$, and that $\pi_{+}$ and $\pi_{-}$ are both empty. This means $p=2$ is type 1, and so $\alpha_2\alpha_3$ is either $a\bar{a}$ or $b\bar{a}$. Only the first case is consistent with us previously finding $\alpha_2=a$, so $\alpha_3=\bar{a}$.

For $p=3$, we see that $p$ occurs before $p+1$, $\pi_{-}$ is non-empty, and $\pi_{+}$ is empty. This mean $p=3$ is type 3, and so $\alpha_3\alpha_4$ is either $\bar{a}\bar{a}$ or $\bar{b}\bar{a}$. Only the first choice is consistent with $\alpha_3=\bar{a}$, so $\alpha_4=\bar{a}$.

Thus, we see that $\beta(2314)=aa\bar{a}\bar{a}$.

\end{example}

Since this map is a bijection, as knowing the type for each $\alpha_p\alpha_{p+1}$ uniquely determines $\alpha$, knowing the type for each $p$ is enough to recover what the original alternating Baxter permutation is.

Thus, we will find it convenient to encode elements of $\mathrm{AltBax}_{2n}$ and $\mathrm{Shuffle}_{2n}$ as words of length $n$ on the letter set $\{1,2\ldots,8\}$, where the $p^{th}$ letter indicates the type of $p$ in $w$ (resp. the type of $\alpha_p\alpha_{p+1}$).

\begin{example}
For $\pi=2314$, as $p=1$ was the type 7, $p=2$ was the type 1, and $p=3$ was the type 3, we would encode this element as $713$.
\end{example}

\begin{theorem}
The bijection between $\Altbax{2n}$ and $Y_{n}^{(3)}(22)$ is equivariant with respect to conjugation by $w_0$ and evacuation.
\end{theorem}

\begin{proof}

The bijection of Dulucq and Guibert from $\mathrm{AltBax}_{2n}$ to $Y_n^{(3)}(22)$ is a composition of a map $\beta$ from $\mathrm{AltBax}_{2n}$ to $\mathrm{Shuffle}_{2n}$ and the previously defined $f$ from $\mathrm{Shuffle}_{2n}$ to $Y_n^{(3)}(22)$. So we need to show that if $f(\beta(w))=x$, then $f(\beta(w_0 w w_0))=\mathrm{evac}(x)$.

We note that the elements in the intermediate set, $\mathrm{Shuffle}_{2n}$, have no natural involution associated to them. However, we can define an involution on $\mathrm{Shuffle}_{2n}$ by mapping it bijectively to a set with a natural involution, doing an involution there, and then mapping it back.

This gives us two possible ways of defining an involution on $\mathrm{Shuffle}_{2n}$ that are not obviously the same. One option is $\alpha\mapsto\beta(w_0 \beta^{-1}(\alpha) w_0)$, induced from $\mathrm{AltBax}_{2n}$. The other is $\alpha\mapsto f^{-1}(\mathrm{evac}(f(\alpha))$, induced from $Y_n^{(3)}(22)$. Proving equivariance is equivalent to showing that these two induced involutions are the same.

First, we describe the involution on $\mathrm{Shuffle}_{2n}$ induced from conjugation by $w_0$ on the alternating Baxter permutations. 

If we know what type $p$ is in the original word $\pi$, we can readily figure out the type of $n-p$ in the involuted word, $\hat{\pi}=w_0 \pi w_0$, as:

\begin{itemize}
\item $p$ appears before $p+1$ in $\pi$ iff $n-p$ appears before $n+1-p$ in $\hat{\pi}$.
\item $\pi_{-}$ is empty iff $\hat{\pi}_{+}$ is empty.
\item $\pi_{+}$ is empty iff $\hat{\pi}_{-}$ is empty. 
\end{itemize}

Thus, if $p$ is of type 1,2,3,4,5,6,7,8 in the original word, then $n-p$ will be of type 1,3,2,4,5,7,6,8 (respectively)

This means that the involution on $\mathrm{Shuffle}_{2n}$ induced from $\mathrm{AltBax}_{2n}$ corresponds to reversing the encoded word, swapping $2$'s and $3$'s, and swapping $7$'s and $6$'s.

\begin{example}
The encoded word for $\pi=2314$ is $713$, so the encoded word for $w_0 \pi w_0 = 1423$ should be $216$. Sure enough, in $1423$, $p=1$ is of type $2$, $p=2$ is of type $1$, and $p=3$ is of type $6$.
\end{example}

Now, we consider the relationship between $\mathrm{Shuffle}_{2n}$ and $Y_n^{(3)}(22)$. Say we have $f(\alpha)=a$. Each of the $2n$ letters of $\alpha$ corresponds to one of the $2n$ instances of $1$ and $3$ in $a$, and additionally keeps track of whether or not that instance of $1$ or $3$ is preceded by a $2$.


Let $\hat{\alpha}=f^{-1}(\mathrm{evac}(f(\alpha))$, representing the involution on $\mathrm{Shuffle}_{2n}$ induced by $Y^{(3)}(22)$.

\begin{prop}
\label{Prop1}
$\hat{\alpha}_{2n+1-p}$ corresponds to a $1$ (resp. $3$) if and only if $\alpha_p$ corresponds to a $3$ (resp. $1$).
\end{prop}

\begin{proof}
Doing evacuation on a Yamanouchi word corresponds to reversing the word, and swapping $1$'s and $3$'s. So if the $p^{th}$ occurrence of either a $1$ or a $3$ in $x\in Y^{(3)}(22)$ is a $1$ (resp. $3$) the $(2n+1-p)^{th}$ occurence of either a $1$ or a $3$ in $\mathrm{evac}(x)$ will be a $3$ (resp. $1$). So if $\hat{\alpha}=f^{-1}(\mathrm{evac}(f(\alpha))$, then $\hat{\alpha}_{2n+1-p}$ will be either $\bar{a}$ or $\bar{b}$ (resp. $a$ or $b$).
\end{proof}

\begin{prop}
\label{Prop2}
The $1$ or $3$ in $\hat{\alpha}_{2n+1-p}$ will be preceded by a $2$ if and only if $\alpha_{p+1}$ corresponds to something that is preceded by a $2$.
\end{prop}

\noindent This easily follows from the fact that evacuation reverses the word.

Say we know whether $\alpha_p$ corresponds to a $1$ or a $3$, and what $\alpha_{p+1}$ is. There are $8$ different cases, and for each case there are two possibilities for what $\alpha_p \alpha_{p+1}$ could be (depending on whether or not the $1$ or $3$ in $\alpha_p$ is preceded by a $2$ or not). One can see that these are exactly the $8$ different types from Figure~\ref{chart}.

\begin{example}
Say we know that $\alpha_p$ corresponds to a $1$, $\alpha_{p+1}$ corresponds to a $1$, and $\alpha_{p+1}$ is preceded by a $2$. Then $\alpha_p\alpha_{p+1}$ could be $ab$ or $bb$, corresponding to type 2.
\end{example}

By Proposition~\ref{Prop1}, we can determine whether $\hat{\alpha}_{2n-p}$ and $\hat{\alpha}_{2n+1-p}$ correspond to $1$'s or $3$'s. By Proposition~\ref{Prop2}, we can also determine whether $\hat{\alpha}_{2n+1-p}$ is preceded by a $2$ or not. Thus, we can determine which of the 8 different cases from Figure~\ref{chart} $\hat{\alpha}_{2n-p}\hat{\alpha}_{2n+1-p}$ corresponds to. One can check case-by-case that we get the same correspondence as before.

\begin{example}
Say we know that $\alpha_p\alpha_{p+1}$ is of type $2$ as in the previous example. Then $\alpha_p$ corresponding to a $1$ means $\hat{\alpha}_{2n+1-p}$ corresponds to a $3$. And $\alpha_{p+1}$ corresponding to a $1$ means $\hat{\alpha}_{2n-p}$ corresponds to a $3$. Finally, $\alpha_{p+1}$ being preceded by a $2$ means $\hat{\alpha}_{2n+1-p}$ is preceded by a $2$. Thus, $\hat{\alpha}_{2n-p}\hat{\alpha}_{2n+1-p}$ could be $\bar{a}\bar{a}$ or $\bar{b}\bar{a}$, corresponding to type $3$.
\end{example}


\end{proof}

\begin{corollary}
The bijection between $\mathrm{Twin}_n$ and $\mathrm{Bax}_n$ is equivariant.
\end{corollary}

\begin{corollary}
Let $n=k+\ell+1$. The bijection from ($k$,$\ell$)-Baxter permutations to ($k$,$\ell$)-Baxter tableaux  is equivariant with respect to conjugation by $w_0$ and evacuation.
\end{corollary}

\begin{remark}
Although it is not our primary interest, the following corollary also allows us to count how many alternating Baxter permutations of length $2n$ and standard $3\times n$ Young tableaux avoiding 22 are fixed under evacuation.
\end{remark}

\begin{corollary}
 The number of alternating Baxter permutations of length $2n$ fixed under conjugation by $w_0$ and the number of $3\times n$ standard Young tableaux with no consecutive entries in the middle row fixed under evacuation are both equal to $c_n$, the Catalan number.
\end{corollary}

\begin{proof}We have an equivariant bijection between pairs of complete binary trees and these two objects, so it suffices to count how many pairs of complete binary trees are fixed under their involution. The involution on pairs of trees is given by reflecting each tree horizontally and then swapping the order of the pair. So a pair fixed under involution is completely determined by the first tree, and it is clear that each complete binary tree yields a pair fixed under involution (by pairing a tree $T$ with a reflected copy of itself). So there are as many pairs of complete binary trees fixed under involution as there are pairs of complete binary trees, of which there are $c_n$.

\end{proof}

In Figure~\ref{altbaxmap}, one can see that the original alternating Baxter permutation is fixed under conjugation by $w_0$, and that the right tree is the mirror image of the left tree.

 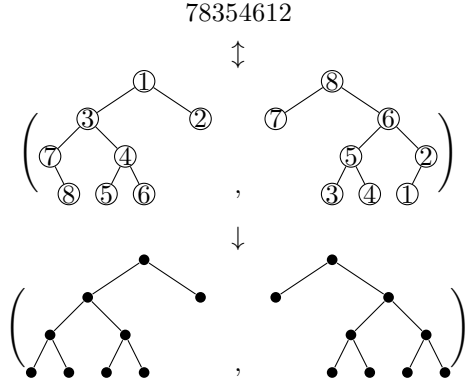
\begin{figure}
 \centering
 $\begin{array}{c}
 78354612 \\
 \updownarrow \\
 \raisebox{\height}{$\Bigg($}
 \begin{tikzpicture}
 [inner sep=0mm, level distance=5mm, level 3/.style={sibling distance=5mm}, level 2/.style={sibling distance=10mm}] 
    \node[circle,draw] {1}  	
      child {node[circle,draw] {3}
 	child {node[circle,draw] {7}
 	  child[missing] {node[circle,draw] {9}}
 	  child {node[circle,draw] {8}}
 	}
 	child {node[circle,draw] {4}
 	  child {node[circle,draw] {5}}
 	  child {node[circle,draw] {6}}
 	}
     }
      child {node[circle,draw] {2}
      };
 \draw (1.25,-1.5) node {,};
 \begin{scope}[xshift=25mm]
 [inner sep=0mm, level distance=5mm, level 3/.style={sibling distance=5mm}, level 2/.style={sibling distance=10mm}, level 1/.style{sibling distance=10mm},xshift=100cm] 
    \node[circle,draw] {8}  	
      child {node[circle,draw] {7}}
 	child {node[circle,draw] {6}
 	  child {node[circle,draw] {5}
 	    child{node[circle,draw] {3}}
 	    child{node[circle,draw] {4}}
 	  }
 	  child {node[circle,draw] {2}
 	    child {node[circle,draw] {1}}
 	    child[missing] {node {0}}
 	  }
 	};
 \end{scope}
 \end{tikzpicture}
 \raisebox{\height}{$\Bigg)$}\\
 \downarrow \\
 \raisebox{\height}{$\Bigg($}
 \begin{tikzpicture}
 [inner sep=.5mm, level distance=5mm, level 3/.style={sibling distance=5mm}, level 2/.style={sibling distance=10mm}] 
    \node[fill,circle] {}  	
      child { node[fill,circle] {}
 	child { node[fill,circle] {}
 	  child { node[fill,circle] {} }
 	  child { node[fill,circle] {} }
 	}
 	child { node[fill,circle] {}
 	  child { node[fill,circle] {} }
 	  child { node[fill,circle] {} }
 	}
      }
      child { node[fill,circle] {}
      };
 \draw (1.25,-1.5) node {,};
 \begin{scope}[xshift=25mm]
 [inner sep=0mm, level distance=5mm, level 3/.style={sibling distance=5mm}, level 2/.style={sibling distance=10mm}, level 1/.style{sibling distance=10mm},xshift=100cm] 
    \node[fill,circle] {}  	
      child { node[fill,circle] {} }
 	child { node[fill,circle] {}
 	  child { node[fill,circle] {}
 	    child{ node[fill,circle] {} }
 	    child{ node[fill,circle] {} }
 	  }
 	  child { node[fill,circle] {}
 	    child { node[fill,circle] {} }
 	    child { node[fill,circle] {} }
 	  }
 	};
 \end{scope}
 \end{tikzpicture}
 \raisebox{\height}{$\Bigg)$}\\
 \end{array}$
 \caption{Map from alternating Baxter permutations of length $2n$ to pairs of complete binary trees with $2n+1$ nodes.}
 \label{altbaxmap}
 \end{figure}

\subsection{Objects (A) and (B)}

Next, we show the equivalence of objects (A) and (B)., using the following fact.

\begin{prop}[Law and Reading~\cite{LR}, Corollary 4.2]

A permutation $w$ lies in $\mathrm{Bax}_n$ if and only if $w^{-1}$ lies in $\mathrm{Bax}_n$.
\end{prop}

If one were dealing with regular pattern avoidance, this would be trivial, because a permutation $w$ contains an instance of 2413 (resp. 3142) if and only if $w^{-1}$ contains an instance of 3142 (resp. 2413). However, one has to do some extra work to check that the analogous statement holds when one has the extra adjacency conditions of vincular patterns.

\begin{prop}
The map $w\mapsto w^{-1}$ gives a bijection between Baxter permutations with $k$ descents and Baxter permutations with $k$ inverse descents that commutes with conjugation by $w_0$.
\end{prop}

\begin{proof}
Conjugation by $w_0$ commutes with $w\mapsto w^{-1}$, since $w_0^{-1}=w_0$.
\end{proof}

While this result on its own is elementary, it is important because the previous Baxter families all had statistics that naturally corresponded to ascents/descents, while the remaining Baxter families all have statistics that will correspond to inverse ascents/inverse descents.


\subsection{Objects (B) and (C)}

There is another class of Baxter objects known as \emph{twisted Baxter permutations}. While Baxter permutations avoid the patterns 3-14-2 and 2-41-3, twisted Baxter permutations avoid the patterns 3-41-2 and 2-41-3. Even though the two pairs of patterns look similar, it is not immediately obvious that they should be so closely related. Section 8 of Law and Reading's paper \cite{LR} provides a bijection between the two that relies on looking at fibers of the lattice congruence $\Theta_{3412}$ on the weak order for $S_n$~\cite{LR}.

\begin{definition}
For $w\in S_n$, let $\mathrm{Inv}(w)$ be the set of inversions, or pairs $(w_i,w_j)$ with $1\leq i<j\leq n$ such that $w_j<w_i$. We say that $u\leq v$ in the weak order if $\mathrm{Inv}(u)\subseteq\mathrm{Inv}(v)$.
\end{definition}

\begin{theorem}[Corollary 3.1.4,\cite{BjornerBrenti}]
The covering relations for the weak order on $S_n$ come precisely from the pairs of permutations that differ only in two adjacent entries.
\end{theorem}

We need the following proposition, which follows immediately from Proposition 8.1 in their paper. 

\begin{definition}
A 3-14-2 $\rightarrow$ 3-41-2 move on a permutation is an action that takes an instance of the pattern 3-14-2, and switches the adjacent entries in the middle so the subsequence corresponds to an instance of the pattern 3-41-2. That is to say, if $w=w_1\ldots w_n$ has an instance of 3-14-2 corresponding to the subsequence $w_i w_j w_{j+1} w_k$, then a 3-14-2 $\rightarrow$ 3-41-2 move would send $w$ to $w_1\ldots w_{j-1}w_{j+1}w_jw_{j+2}\ldots w_n$.
\end{definition}

\begin{prop}
Given a twisted Baxter permutation, it will be the maximal element in its fiber over $\Theta_{3412}$, the corresponding Baxter permutation will be the unique minimal element, and the fiber will consist of all permutations attainable from the twisted Baxter permutation by making any sequence of (3-14-2 $\rightarrow$ 3-41-2) moves.
\end{prop}

\begin{figure}[h]
\label{3412fiber}
\centering
\begin{tikzpicture}
\node (a) at (0,7) {4567123};
\node (b) at (0,6) {4561723};
\node (c) at (-1,5) {4516723};
\node (d) at (1,5) {4561273};
\node (e) at (-1,4) {4156723};
\node (f) at (1,4) {4516273};
\node (g) at (0,3) {4156273};
\node (h) at (0,2) {4152673};
\node (i) at (0,1) {4125673};
\draw (a) -- (b) -- (c) -- (e) -- (g) -- (f) -- (c);
\draw (b) -- (d) -- (f) -- (g) -- (h) -- (i);
\end{tikzpicture}
\caption{Fiber of the congruence $\theta_{3412}$ with the Baxter permutation 4567123 as its maximal element, and the twisted Baxter permutation 4125673 as its minimal element.}
\end{figure}
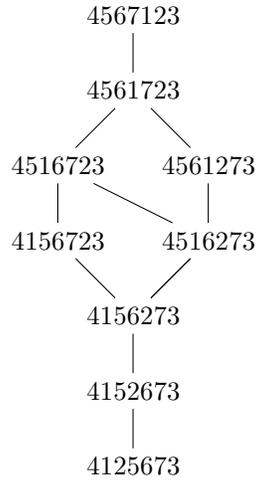

\begin{corollary}
The number of twisted Baxter permutations of length $n$ with $k$ inverse descents is equal to the number of Baxter permutations of length $n$ with $k$ inverse descents.
\end{corollary}

\begin{proof}
The moves that get us from a twisted Baxter permutation to a Baxter permutation will never change the number of inverse descents. Swapping the elements playing the role of 1 and 4 in adjacent positions will never change the relative order of $i$ and $i+1$ for any $i$.
\end{proof}

\begin{corollary}
A twisted Baxter permutation and its corresponding Baxter permutation are each fixed under conjugation by $w_0$ if and only if their common fiber is fixed under conjugation by $w_0$.
\end{corollary}

\begin{proof}
Since the fibers of this congruence can be described as the orbit of all possible (3-14-2 $\leftrightarrow$ 3-41-2) moves, conjugation by $w_0$ will map fibers to fibers.
\end{proof}



\subsection{Objects (C) and (F)}

\begin{definition}
A \emph{diagonal rectangulation} of size $n$ is a subdivision of an $n\times n$ square into $n$ rectangles (with lattice points for corners) such that the interior of every rectangle intersects a fixed diagonal of the square.
\end{definition}

\begin{figure}[!h]
\centering
\tikzstyle{my help lines}=[gray,thick,dashed]
\begin{tikzpicture}[scale=.65]
 \draw (0,0) rectangle (3,2);
 \draw (0,2) rectangle (1,4);
 \draw (1,2) rectangle (3,4);
 \draw (3,0) rectangle (4,4);
 \draw[style=my help lines] (0,4) -- (4,0);
\end{tikzpicture}\caption{A diagonal rectangulation of size $n=4$}
\end{figure}
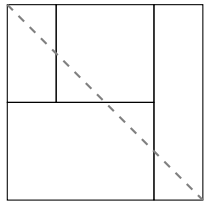


We next check to see that the bijection between twisted Baxter permutations and diagonal rectangulations given in Section 6 of Law and Reading~\cite{LR} preserves the indicated statistic, and will equivariantly take conjugation by $w_0$ to $180^{\circ}$ rotation. We again have the intermediate object of pairs of twin trees. 

The map from twisted Baxter permutations to pairs of twin trees used by Law and Reading is equivalent to $\mathrm{Extend}^{\times 2}(w^{-1})$ (they respectively call these the \emph{upper} and \emph{lower planar binary trees}). Conjugation by $w_0$ on twisted Baxter permutations will correspond to the same involution on pairs of trees defined in Definition~\ref{treeinvdef}. Also, if a twisted Baxter permutation $w$ has $k$ inverse ascents, $w^{-1}$ will have $k$ ascents, and $\mathrm{incr}(w^{-1})$ will have $k$ left leaves (excluding the left-most one), preserving the statistic.

A diagonal rectangulation is made by gluing the two trees together. In particular, one draws the trees so that all the leaves are evenly spaced on the lowest level, and all intersections make right angles. Then the twin tree condition guarantees that if we turn the left tree upside-down, it will match up with the right tree to form a diagonal rectangulation (see Figure~\ref{tbax}). 

It is then obvious that the involution on pairs of trees corresponds to $180^{\circ}$ rotation on diagonal rectangulations, and that $k$ left leaves in the left tree (excluding the left-most one) will correspond to the $k$ vertical intersections with the interior of the diagonal.

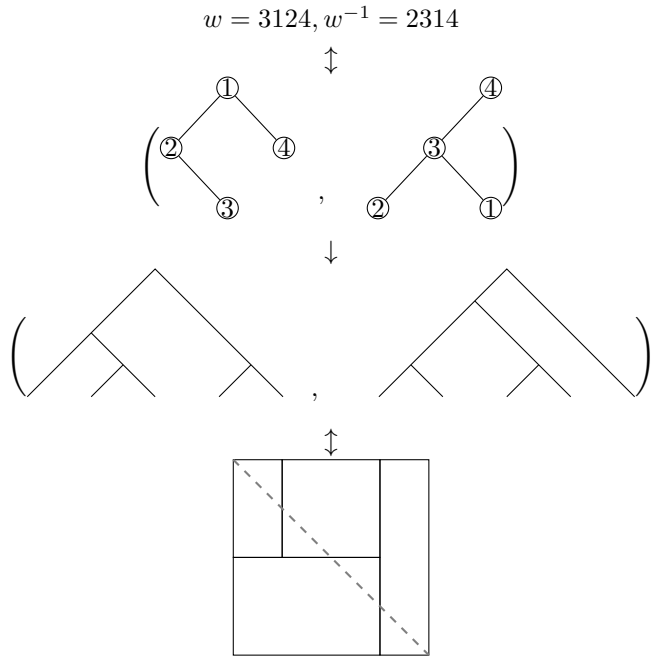
\begin{figure}[!thb]
\centering
$\begin{array}{c}
 w=3124, w^{-1}=2314 \\
\updownarrow \\
\raisebox{\height}{$\Bigg($}
\begin{tikzpicture}
[inner sep=0mm, level distance=8mm] 
    \node[circle,draw] {1}  	
      child {node[circle,draw] {2}
	child[missing] {node[circle,draw] {5}}
	child {node[circle,draw] {3}}
      }
      child {node[circle,draw] {4}
      };
\draw (1.25,-1.5) node {,};
\begin{scope}[xshift=35mm]
[inner sep=0mm, level distance=8mm] 
    \node[circle,draw] {4}  	
      child {node[circle,draw] {3}
	child {node[circle,draw] {2}}
	child {node[circle,draw] {1}}
      }
      child[missing]{node[circle,draw] {5}
	};
\end{scope}
\end{tikzpicture}
\raisebox{\height}{$\Bigg)$}\\
\downarrow \\
\raisebox{\height}{$\Bigg($}

\begin{tikzpicture}[scale=.85]
 \draw (0,0) -- (2,2) -- (4,0) ;
 \draw (1,0) -- (1.5,.5);
 \draw (2,0) -- (1,1);
 \draw (3,0) -- (3.5,.5);
\begin{scope}[xshift=55mm]
 \draw (0,0) -- (2,2) -- (4,0) ;
 \draw (1,0) -- (.5,.5);
 \draw (2,0) -- (2.5,.5);
 \draw (3,0) -- (1.5,1.5);
\end{scope}
\draw (4.5,0) node {,};
\end{tikzpicture}

\raisebox{\height}{$\Bigg)$}\\
\updownarrow \\
\tikzstyle{my help lines}=[gray,thick,dashed]
\begin{tikzpicture}[scale=.65]
 \draw (0,0) rectangle (3,2);
 \draw (0,2) rectangle (1,4);
 \draw (1,2) rectangle (3,4);
 \draw (3,0) rectangle (4,4);
 \draw[style=my help lines] (0,4) -- (4,0);
\end{tikzpicture}\\

\end{array}$
\caption{Map from twisted Baxter permutations to diagonal rectangulations}
\label{tbax}
\end{figure}

\section{Proof of Theorem~\ref{Theorem 2}}
\label{Section 3}

By Theorem \ref{Theorem 1}, if we want to count the number of objects fixed under involution for any Baxter object, we only have to find the number of objects fixed under involution for one family of Baxter objects. This is easiest for Baxter plane partitions. MacMahon gave a closed formula for the generating function of plane partitions inside a box, weighted by number of boxes.

\newpage

\begin{theorem}[\cite{Stanley},Theorem 7.21.7]
 Fix $a$, $b$, and $c$, and let $|\pi|$ be the total number of boxes in a plane partition. Then
\begin{equation}
\label{MacMahon} 
  \sum_{\pi} q^{|\pi|}=\prod_{\substack{i=1,\ldots,a \\ j=1,\ldots,b \\ k=1,\ldots,c }} \frac{[i+j+k-1]_q}{[i+j+k-2]_q}
\end{equation}
 where $\pi$ runs over all plane partitions that fit in an $a\times b\times c$ box.

\end{theorem}

 We will write the above sum as $M(a,b,c;q)$. One can check that for Baxter plane partitions, this gives the previously defined $q$-analogue of $\Theta_{k,\ell}$.

\begin{corollary}
\label{generatingfunction} 
\begin{equation}
M(k,\ell,3;q)=\sum_{\pi} q^{|\pi|} = \frac{\qbin{n+1}{k}{q}\qbin{n+1}{k+1}{q}\qbin{n+1}{k+2}{q}}{\qbin{n+1}{1}{q}\qbin{n+1}{2}{q}} = \Theta_{k,\ell}(q)
\end{equation} 
 
\noindent where $\pi$ runs over all plane partitions that fit in a $k\times\ell\times 3$ box.

\end{corollary}

In particular, this tells us that $\Theta_{k,\ell}(q)$ is indeed a polynomial with symmetric, non-negative integer coefficients, which is not immediately obvious from the definition.

Additionally, we have the following theorem of Stembridge.

\begin{theorem}[Stembridge, Example 2.1, \cite{Stembridge}]
The number of self-complementary plane partitions that fit inside an $a\times b\times c$ box is $M(a,b,c;-1)$.
\end{theorem}




By setting $a=k$, $b=\ell$, and $c=3$, we get the following result.

\begin{theorem}
\label{qequalsminusone}
$\Theta_{k,\ell}^\circlearrowleft=[\Theta_{k,\ell}(q)]_{q=-1}$
\end{theorem}

Although Theorem \ref{qequalsminusone} follows from Stembridge's result without any further computation, it turns out that it agrees with formulas for $\Theta_{k,\ell}^{\circlearrowleft}$ given previous by Felsner, Fusy, Orden, and Noy~\cite{BBFRO}, after correcting one of the cases of their formula, and applying a hypergeometric summation, as we explain next.

\begin{theorem}
\label{FFON}
 \begin{enumerate}
  \item If $k$ and $\ell$ are odd, then $\Theta_{k,\ell}^\circlearrowleft=0$
  \item If $k$ and $\ell$ are even, with $k=2\kappa$ and $\ell=2\lambda$, then for $N=\kappa+\lambda$, 
  \begin{align}
  \label{eveneven}
  \Theta_{2\kappa,2\lambda}^\circlearrowleft & = \sum_{r\geq 0} \frac{2r^3}{(N+1)(N+2)^2} \binom{N+2}{\kappa+1}\binom{N+2}{\kappa-r+1}\binom{N+2}{\kappa+r+1} & \notag \\ 
  & =\frac{\binom{N+1}{\kappa}\binom{N+1}{\kappa +1}\binom{N}{\kappa}}{(N+ 1)}
  \end{align}
  \item If $k$ is odd and $\ell$ is even \footnote{This case corrects Proposition 7.4, part iii in Felsner, Fusy, Orden, and Noy}, with $k=2\kappa +1$ and $\ell=2\lambda$, then for $N=\kappa+\lambda$,  
  \begin{align}
  \label{evenodd}
   \Theta_{2\kappa +1,2\lambda}^{\circlearrowleft}  & =  \Theta_{2\kappa,2\lambda}^\circlearrowleft + \sum_{r\geq 1} \frac{(\lambda-r+1)r(r+1)(2r+1)}{(\kappa +2+r)(N+1)(N+ 2)^2} \binom{N+2}{\kappa+1}\binom{N+2}{\kappa-r+1}\binom{N+2}{\kappa+r+1} & \notag\\
    & =\frac{\binom{N+1}{\kappa}\binom{N +1}{\kappa }\binom{N+1}{\kappa +1}}{(N+1)}
  \end{align}
  \item If $k$ is even and $\ell$ is odd, with $k=2\kappa$ and $\ell=2\lambda+1$, then $$\Theta_{2\kappa ,2\lambda+1}^{\circlearrowleft}=\Theta_{2\lambda +1,2\kappa}^{\circlearrowleft},$$ which is the same as (\ref{evenodd}) with $\kappa$ and $\lambda$ switched.
 \end{enumerate}
\end{theorem}

Before embarking on the proof, we review the approach used by Felsner, Fusy, Orden, and Noy. They counted non-intersecting triples of lattice paths as in \eqref{latticepoints} fixed under rotation. The rotation will be about the point $(\frac{k}{2},\frac{\ell}{2})$, and a rotationally invariant Baxter path will be uniquely determined by what it does below the line $x+y=\frac{k+\ell}{2}$. In \cite{BBFRO}, the authors show that all rotationally invariant Baxter paths arise from triples of lattice paths from $A_1$, $A_2$, and $A_3$ to specific points below the line $x+y=\frac{k+\ell}{2}$, which depend on the parity of $k$ and $\ell$, and also a parameter $r$. These triples of lattice paths can be counted by the Gessel-Viennot-Lindstrom lemma~\cite{GVL}, and they obtain their formula by summing the resulting expressions over all possible parameters $r$ for $k$ and $\ell$ of fixed parity, resulting in the first parts of (\ref{eveneven}) and (\ref{evenodd}).

\begin{proof}[Proof of Theorem~\ref{FFON}]

\vskip .1in
\noindent
\newline
{\sf Proof of Assertion 1.} When $k$ and $\ell$ are both odd, one can easily see $\Theta_{k,\ell}^{\circlearrowleft}=0$. In terms of rotationally invariant lattice paths, the central point of rotation will have two non-integral coordinates. In order for the path from $A_2$ to meet up with itself upon rotation, it would have to go through this point, but lattice paths always have at least one integral coordinate. One can also look at the plane partition model, where the $k\times\ell\times 3$ box has an odd number of boxes, so the size of any plane partition in that box must have opposite parity of its complement. Correspondingly, we check that $[\Theta_{k,\ell}(q)]_{q=-1}=0$. In this case, the denominator of \eqref{defofqanalog} only has one factor of $1+q$, coming from $\qbin{n+1}{1}{q}$, whereas the numerator will have two factors of $(1+q)$, coming from each of $\qbin{n+1}{k}{q}$ and $\qbin{n+1}{k+2}{q}$.

\vskip .1in
\noindent
{\sf Proof of Assertion 2.} When $k=2\kappa$ and $\ell=2\lambda$ are both even, the resulting summation in (\ref{eveneven}) is (after factoring out a constant) the hypergeometric series
\[{}_5 F_4 \left[ \left.
\begin{matrix}
 2 & 2 & 2 & -\kappa, & -\lambda \\
  & 1 & 1 & \kappa+3 & \lambda+3 \\
\end{matrix}
\right| 1 \right],   \]

\noindent where we recall that \[{}_r F_s \left[ \left.
\begin{matrix}
 a_1 & a_2 & \ldots & a_r  \\
   b_1 & b_2 & \ldots & b_s \\
\end{matrix}
\right| z \right]=\sum_{n=1}^{\infty} \frac{(a_1)_n\ldots (a_r)_n}{(b_1)_n\ldots (b_s)_n}z^n  , \]
for $(x)_n=x(x-1)\ldots (x-n+1)$.

\noindent This can be evaluated using the formula for a well-poised ${}_5 F_4$ ~\cite[(4.4.1) p.27]{Bailey},

\begin{align}
 & {}_5 F_4 \left[ \left. \begin{matrix} a & \frac{1}{2} a+1 & c & d & e \\ & \frac{1}{2} a & 1+a-c & 1+a-d & 1+a-e\\ \end{matrix} \right| 1 \right]\label{5F4}\\ 
 & =\frac{\Gamma(1+a-c)\Gamma(1+a-d)\Gamma(1+a-e)\Gamma(1+a-c-d-e)}{\Gamma(1+a)\Gamma(1+a-d-e)\Gamma(1+a-c-e)\Gamma(1+a-c-d)}. \notag
\end{align}

\noindent By choosing $a=c=2$, $d=-\kappa$, and $e=-\lambda$, one can check this gives (\ref{eveneven}).

\vskip .1in
\noindent
{\sf Proof of Assertion 3.} When $k=2\kappa+1$ is odd and $\ell=2\lambda$ is even, the summation is (again, after factoring out a constant) the hypergeometric series 
\begin{equation*}
{}_4 F_3 \left[ \left.
\begin{matrix}
 3, & 5/2, & 1-\lambda, & -\kappa; \\
  & 3/2, & \lambda +3, & \kappa+4  \\
\end{matrix}  
\right| 1 \right]  .
\end{equation*}

\noindent This can also be evaluated using \eqref{5F4}, but with choice of parameters $a=3$, $c=1-\lambda$, $d=-\kappa$, and $e=2$ (note that as $e=1+a-e$, the ${}_5 F_4$ reduces to a ${}_4 F_3$), and one can check that this gives (\ref{evenodd}). 

\vskip .1in
\noindent
{\sf Proof of Assertion 4.} We exploit natural symmetry that forces $\Theta_{k,\ell}^\circlearrowleft=\Theta_{\ell,k}^\circlearrowleft$.

\end{proof}

\section{A possible $q$-Baxter number}

The formula (\ref{defofqanalog}) gives a meaningful $q$-analog for $\Theta_{k,\ell}$, and we would like to extend it to a $q$-analog for Baxter numbers. A natural way in which one can generalize is inspired by placing Baxter numbers within the family of Hoggatt sums~\cite{FielderAlford}. 

Let $H(m,k,\ell)$ be the number of plane partitions that fit in a $k\times\ell\times m$ box, which we will call the \emph{MacMahon numbers}. Via MacMahon's plane partition formula given in \eqref{MacMahon}, these can be simply expressed as \begin{equation}H(m,k,l)=\frac{\prod_{i=0}^{m-1} \binom{k+\ell}{k+i}}{\prod_{j=1}^{m-1} \binom{k+\ell}{j}}.\label{MacMahonproduct}\end{equation} We also consider a natural $q$-shift of MacMahon's formula, \begin{equation}H(m,k,\ell;q)=q^{m\binom{k+1}{2}}\sum_{\pi} q^{|\pi|},\label{qMacMahon}\end{equation} where we sum over all plane partitions $\pi$ in a $k\times\ell\times m$ box. Note that $H(m,k,\ell,1)=H(m,k,\ell)$.

We will define the \emph{Hoggatt sum} and \emph{$q$-Hoggatt sum} to respectively be

\[H(n,m)=\sum_{k+\ell=n} H(m,k,\ell)\]

\[H(n,m;q)=\sum_{k+\ell=n} H(m,k,\ell;q)\]


\begin{prop}
 $H(n,m;q)$ has symmetric coefficients as a polynomial in $q$.
\end{prop}


\begin{proof}
One can easily check that the degree of $H(n,m,\ell;q)$ will always be $m\binom{n+1}{2}$. Define an involution on the set of plane partitions that contribute to $H(n,m;q)$ by pairing $\pi$ in the $k\times (n-k)\times m$ box with $\bar{\pi}^c$ in the $(n-k)\times k\times m$ box, where $\bar{\pi}$ is the plane partition in $(n-k)\times k\times m$ box naturally identified with $\pi$. It suffices to check that the degrees of the contributions from $\pi$ and $\bar{\pi}^c$ add up to $m\binom{n+1}{2}$, the degree of the polynomial. The contribution from $\pi$ has degree $|\pi|+m\binom{k+1}{2}$, and the contribution from $\bar{\pi}^c$ has degree $|\bar{\pi}^c|+m\binom{n-k+1}{2}=(mn(n-k)-|\pi|)+m\binom{n-k+1}{2}$, and these two terms indeed add up to $m\binom{n+1}{2}$.
 
\end{proof}

For $m=1$, the MacMahon numbers are $H(1,k,\ell)=\binom{k+\ell}{k}$, the \emph{binomial coefficients}, with meaningful $q$-analogue $$H(1,k,\ell;q)=q^{\binom{k+1}{2}}\qbin{k+\ell}{k}{q}.$$ The Hoggatt sum is $H(n,1)=2^n$, and the $q$-Hoggatt sum is $$H(n,1;q)=\sum_{k=0}^n q^{\binom{k+1}{2}}\qbin{n}{k}{q}=(-q;q)_{n+1}=(1+q)(1+q^2)\ldots(1+q^n).$$



For $m=2$, the MacMahon numbers are $H(2,k,\ell)=\frac{1}{k+\ell}\binom{k+\ell}{k}\binom{k+\ell}{k+1}$, the \emph{Narayana numbers}, with $q$-analogue $$H(k,\ell,2;q)=\frac{q^{k^2+k}}{[k+\ell]_q}\qbin{k+\ell}{k}{q}\qbin{k+\ell}{k+1}{q}$$ the $q$-Narayana numbers~\cite{qnarayana}. The Hoggatt sum is $H(n,2)=\frac{1}{n+1}\binom{2n}{n}$, the Catalan number, whereas the $q$-Hoggatt sum is $$H(n,2;q)=\frac{1}{[n+1]_q}\qbin{2n}{n}{q},$$ the \emph{$q$-Catalan number}~\cite{qnarayana}.



For $m=3$, the MacMahon numbers are $H(3,k,\ell)=\Theta_{k,\ell}$, with $q$-analogue $H(3,k,\ell;q)=\Theta_{k,\ell}(q)$. The Hoggatt sums are $H(n,3)=B(n+1)$, the Baxter number, and so one might consider the third $q$-Hoggatt sum \begin{equation}H(n,3;q)=\sum_{k+\ell=n} q^{3\binom{k+1}{2}}\Theta_{k,\ell}(q) \label{qbaxtereqn} \end{equation} as a $q$-Baxter number. However, we do not know if it has any nice combinatorial interpretations like the cases $m=1,2$. Also, it is not known whether $H(4,k,\ell;q)$ or $H(n,4;q)$ has any nice combinatorial interpretation outside of plane partitions, as there is for $m\leq 3$.

\section{Descent Generating Function}

Not much has been said about the generating function for Baxter permutations with respect to descents, 

\[B(n,t)=\sum_{k=0}^{n-1} \frac{\binom{n+1}{k}\binom{n+1}{k+1}\binom{n+1}{k+2}}{\binom{n+1}{1}\binom{n+1}{2}}t^k, \]

\noindent so we will make a few brief comments. First, we are able to show that $B(n,t)$ is real rooted by using the theory of multiplier sequences. 

\begin{definition}
Say that $\{a_k\}_{k\geq 0}$ is a multiplier sequence if for every polynomial $\sum_{k=0}^n b_k t^k$ with all real roots, $\sum_{k=0}^n a_k b_k t^k$ also has all real roots.
\end{definition}

Multiplier sequences satisfy some basic properties, which are outlined in Craven and Csordas~\cite{Craven}.

\begin{prop}
\label{multseq}
Let $\{a_k\}_{k\geq 0}$ be a multiplier sequence.

\begin{itemize}
 \item[(a)] $\sum_{k=0}^n a_k t^k$ is real rooted.
 \item[(b)] $\{a_{k+r}\}_{k\geq 0}$ for $r\geq 0$ is also a multiplier sequence.
 \item[(c)] If $\{b_k\}_{k\geq 0}$ is another multiplier sequence, then $\{a_k b_k\}_{k\geq 0}$ is also a multiplier sequence.
 \item[(d)] If $\{a_r,a_{r+1},\ldots,a_{r+s}\}$ is a segment of a multiplier sequence, and $\sum_{k=0}^n b_k t^k$ is real rooted with $n\leq s$, then $\sum_{k=0}^n a_{r+s-k} b_k t^k$ is real rooted
\end{itemize}
\end{prop}

\begin{prop}
$B(n,t)$ has only real roots.
\end{prop}

\begin{proof}

It suffices to check that $$\sum_{k=0}^{n-1} \frac{t^k}{k!(k+1)!(k+2)!(n+1-k)!(n-k)!(n-k-1)!}$$ is real rooted, as $B(n,t)$ is this sum times the constant $2(n+1)!n!(n-1)!$.

It is well known that $\{\frac{1}{k!}\}$ is a multiplier sequence (Theorem 2.4.1,~\cite{Brenti}), so we just need to apply a number of the transformations from Proposition~\ref{multseq}. In particular, (b) tells us that $\{\frac{1}{(k+1)!}\}$ and $\{\frac{1}{(k+2)!}\}$ will be multiplier sequence. So by (c), $\{\frac{1}{k!(k+1)!(k+2)!}\}$ is a multiplier sequence, and $\sum_{k=0}^{n-1} \frac{1}{k!(k+1)!(k+2)!} t^k$ is real rooted. Now apply part (d) to the segments $\{\frac{1}{1!},\ldots, \frac{1}{(n-1)!}\}$, $\{\frac{1}{2!},\ldots, \frac{1}{n!}\}$, and $\{\frac{1}{3!},\ldots, \frac{1}{(n+1)!}\}$ to get the desired result.
\end{proof}

We know that if $w=w_1\ldots w_n$ is a Baxter permutation with $k$ descents, then $w w_0=w_n\ldots w_1$ will be a Baxter permutation with $n-k-1$ descents. This tells us that $B(n,t)$ should have symmetric coefficients. Following Gal~\cite{Gal}, one can consider the expansion of polynomials of degree $n$ with symmetric coefficients in terms of the basis $\{t^i(1+t)^{n-2i}\}_{0\leq i\leq \lfloor\frac{n}{2}\rfloor}$.

\begin{definition}

Let $P(t)$ be a polynomial of degree $n$ with symmetric coefficients, and let $P(t)=\sum_{k=0}^{\lfloor\frac{n}{2}\rfloor} \gamma_i t^i(1+t)^{n-2i}$. Then we say that $P(t)$ is \emph{$\gamma$-positive} if $\gamma_i\geq 0$.

\end{definition}

Gal has conjectured that the $h$-vector for any flag simplicial polytope is $\gamma$-positive~\cite{Gal}. Ideally, one would like to be able to find a polytope related to Baxter objects with $B(n,t)$ as its h-vector.

 As $B(n,t)$ is real-rooted with positive, symmetric coefficients, it should be $\gamma$-positive~\cite[Remark 3.1.1]{Gal}. Since this generating function can be seen to be a well-poised ${}_3F_2$, we can apply the well-poised ${}_3F_2$ quadratic tranformation \cite[p.~97]{Bailey}

\[(1-z)^a {}_3 F_2 \left[ \left. \begin{matrix} a & b & c \\ & 2 & 1 \\ \end{matrix} \right| z \right]={}_3 F_2 \left[ \left. \begin{matrix} \frac{1}{2} a & \frac{1}{2}(a+1) & 1+a-b-c \\ & 1+a-b a & 1+a-c \\ \end{matrix} \right| -\frac{4z}{(1-z)^2} \right]. \]

\noindent Setting $z=-t$, $a= -n+1$, $b=-n$, $c=-n-1$ gives

\[B(n,t)=\sum_{k=0}^{\lfloor n/2 \rfloor} \gamma_i t^i(1+t)^{n-2i}\]

\noindent for 

\begin{equation}
\label{gamma}
\gamma_i=\frac{(n+3)_i(1-n)_{2i}}{(1)_i(2)_i(3)_i}.\footnote{Thanks to Dennis Stanton for suggesting this transformation}
\end{equation}

Ideally, one would like a combinatorial interpretation for these $\gamma_i$'s. The Eulerian polynomials, which are the generating function over all permutations with respect to descents, are known to be $\gamma$-positive. Shapiro, Woan, and Getu have shown that the $\gamma_i$'s for Eulerian polynomials count the number of permutations in $S_n$ with no consecutive descents and no final descent. Furthermore, Foata and Sch\"utzenberger have a formula that gives $\gamma_i$ as a weighted count of all permutations of length $n$ with $i$ peaks. The natural restriction of these interpretations from all permutations to Baxter permutations does not give the correct $\gamma_i$'s, however.

Furthermore, a similar identity can be applied when we consider the $q,t$ version of the above identity for $$B(n,t,q)=\sum_{k=0}^{n-1} \frac{\qbin{n+1}{k}{q}\qbin{n+1}{k+1}{q}\qbin{n+1}{k+2}{q}}{\qbin{n+1}{2}{q}\qbin{n+1}{1}{q}}q^{3\binom{k+1}{2}}t^k.$$ Applying the Sears-Carlitz transformation of a terminating well-poised ${}_3\phi_2$ ~\cite[(III.14)]{GasperRahman}\footnote{Again, thanks to Dennis Stanton for suggesting this $q$-analog of the previous transformation} gives the expression

\[B(n,t,q)=\sum_{i=0}^{\lfloor n/2 \rfloor} q^{3\binom{i}{2}}\frac{(q^{n-2i};q)_{2i}(q^{n+3};q)_i}{(q^3;q)_i(q^2;q)_i(q;q)_i}t^{i}\frac{(-tq^{n-2+i};q)_\infty}{(-tq^{2n-3-i};q)_\infty}.\]

As an example, this formula gives the expansion \[B(4,t,q)=1+(q^6+q^5+2 q^4+2 q^3+2 q^2+q+1)t+q^3(q^6+q^5+2 q^4+2 q^3+2 q^2+q+1)t^2+q^9t^3\]
\[=(1+tq^2)(1+tq^3)(1+tq^4)+(q^6+q^5+q^4+q^3+q^2+q+1)t(1+tq^3).\]

\noindent This would suggest that $$\gamma_i(q)=q^{3\binom{i}{2}}\frac{(q^{n-2i};q)_{2i}(q^{n+3};q)_i}{(q^3;q)_i(q^2;q)_i(q;q)_i}$$ is a natural $q$-analog to the $\gamma_i$ defined in (\ref{gamma}). If a combinatorial interpretation is found for the $\gamma_i$'s, one might expect that $\gamma_i(q)$ will give their distribution with respect to some natural statistic.

\section{Acknowledgements}
\label{sec:ack}
This paper extends the results and provides proofs from a previously published FPSAC extended abstract.~\cite{Dilks}

The author would like to thank Nathan Reading, Jang Soo Kim for helping make the connection to plane partitions, Dennis Stanton for help with hypergeometric series, and to his adviser Vic Reiner for suggesting the problem and for numerous helpful discussions.

\newpage 
\bibliographystyle{abbrv}
\bibliography{baxterinvolution3}
\label{sec:biblio}

\appendix
\section{Appendix}

\begin{table}[!b]
\caption{Baxter objects of order $(k,\ell)=(3,0)$}
\label{Table 1}
\centering
$\Theta_{k,\ell}(q)=1$\\
$\Theta_{k,\ell}(1) =1$\\
$\Theta_{k,\ell}(-1)=1$\\
\begin{tabular}[t]{|c|c|c|c|c|c|c|}

 \hline
Baxter & Twisted Baxter  & Baxter & Baxter  & Diagonal  & Baxter Plane  \\
Permutations &  Permutations  & Paths & Tableaux &  Rectangulations &  Partitions  \\

 \hline
 \hline
 $\begin{array}{c} 1234\\ \circlearrowleft\\ \end{array}$  & $\begin{array}{c} 1234\\ \circlearrowleft\\ \end{array}$  & $\begin{array}{c} \begin{tikzpicture}[scale=.2]
 \draw[gray,very thin] (0,0) grid (5,5);
 \draw[very thick] (2,0) -- (5,0);
 \draw[very thick] (1,1) -- (4,1);
 \draw[very thick] (0,2) -- (3,2);
 \filldraw (0,2) circle (2pt)
           (1,1) circle (2pt)
           (2,0) circle (2pt)
           (5,0) circle (2pt)
           (4,1) circle (2pt)
           (3,2) circle (2pt);

\end{tikzpicture} \\ \circlearrowleft \\ \end{array}$  &  $\begin{array}{c} \young(1369,258\eleven,47\ten\twelve) \\ \circlearrowleft \\ \end{array}$ & $\begin{array}{c} \tikzstyle{my help lines}=[gray,thick,dashed]
\begin{tikzpicture}[scale=.25]
 \draw (0,0) rectangle (1,4);
 \draw (1,0) rectangle (2,4);
 \draw (2,0) rectangle (3,4);
 \draw (3,0) rectangle (4,4);
 \draw[style=my help lines] (0,4) -- (4,0);
\end{tikzpicture} \\ \circlearrowleft \\ \end{array}$  & $\begin{array}{c} \emptyset \\ \circlearrowleft \\ \end{array}$ \\

\hline
\end{tabular}
\end{table}

\begin{table}[!thb]
\caption{Baxter objects of order $(k,\ell)=(0,3)$}
\label{Table 4}
\centering
$\Theta_{k,\ell}(q)=1$\\
$\Theta_{k,\ell}(1) =1$\\
$\Theta_{k,\ell}(-1)=1$\\
\begin{tabular}[t]{|c|c|c|c|c|c|c|}

\hline

Baxter & Twisted Baxter  & Baxter & Baxter  & Diagonal  & Baxter Plane  \\
Permutations &  Permutations  & Paths & Tableaux &  Rectangulations &  Partitions  \\\hline
 \hline
 $\begin{array}{c} 4321 \\ \circlearrowleft \\ \end{array}$ & $\begin{array}{c} 4321 \\ \circlearrowleft \\ \end{array}$ & $\begin{array}{c} \begin{tikzpicture}[scale=.2]
 \draw[gray,very thin] (0,0) grid (5,5);
 \draw[very thick] (2,0) -- (2,3);
 \draw[very thick] (1,1) -- (1,4);
 \draw[very thick] (0,2) -- (0,5);
 \filldraw (0,2) circle (2pt)
           (1,1) circle (2pt)
           (2,0) circle (2pt)
           (0,5) circle (2pt)
           (2,3) circle (2pt)
           (1,4) circle (2pt);

\end{tikzpicture} \\ \circlearrowleft \\ \end{array}$ & $\begin{array}{c} {\young(147\ten,258\eleven,369\twelve)} \\ \circlearrowleft \\ \end{array}$ & $\begin{array}{c} \tikzstyle{my help lines}=[gray,thick,dashed]
\begin{tikzpicture}[scale=.25]
 \draw (0,0) rectangle (4,1);
 \draw (0,1) rectangle (4,2);
 \draw (0,2) rectangle (4,3);
 \draw (0,3) rectangle (4,4);
 \draw[style=my help lines] (0,4) -- (4,0);
\end{tikzpicture} \\ \circlearrowleft \\ \end{array}$ & $\begin{array}{c} \emptyset \\ \circlearrowleft \\ \end{array}$ \\

\hline
\end{tabular}

\end{table}

\begin{table}
\caption{Baxter objects of order $(k,\ell)=(2,1)$}
\label{Table 2}
\centering
$\Theta_{k,\ell}(q)=1+q+2q^2+2q^3+2q^4+q^5+q^6$\\
$\Theta_{k,\ell}(1) =10$\\
$\Theta_{k,\ell}(-1)=2$\\
\begin{tabular}[t]{|c|c|c|c|c|c|c|}

 \hline
Baxter & Twisted Baxter  & Baxter & Baxter  & Diagonal  & Baxter Plane  \\
Permutations &  Permutations  & Paths & Tableaux &  Rectangulations &  Partitions  \\
\hline
\hline
 $\begin{array}{c} 1243\\ \updownarrow \\ 2134 \\ \end{array}$ & $\begin{array}{c} 1243\\ \updownarrow \\ 2134 \\ \end{array}$ & $\begin{array}{c} \begin{tikzpicture}[scale=.2]
 \draw[gray,very thin] (0,0) grid (5,5);
 \draw[very thick] (2,0) -- (4,0) -- (4,1);
 \draw[very thick] (1,1) -- (3,1) -- (3,2);
 \draw[very thick] (0,2) -- (2,2) -- (2,3);
 \filldraw (0,2) circle (2pt)
           (1,1) circle (2pt)
           (2,0) circle (2pt)
           (2,3) circle (2pt)
           (4,1) circle (2pt)
           (3,2) circle (2pt);

\end{tikzpicture} \\ \updownarrow \\ \begin{tikzpicture}[scale=.2]
 \draw[gray,very thin] (0,0) grid (5,5);
 \draw[very thick] (2,0) -- (2,1) -- (4,1);
 \draw[very thick] (1,1) -- (1,2) -- (3,2);
 \draw[very thick] (0,2) -- (0,3) -- (2,3);
 \filldraw (0,2) circle (2pt)
           (1,1) circle (2pt)
           (2,0) circle (2pt)
           (4,1) circle (2pt)
           (3,2) circle (2pt)
           (2,3) circle (2pt);

\end{tikzpicture} \\ \end{array}$ &  $\begin{array}{c} {\young(1369,257\eleven,48\ten\twelve)} \\ \updownarrow \\ {\young(1359,268\eleven,47\ten\twelve)} \\ \end{array}$ & $\begin{array}{c} \tikzstyle{my help lines}=[gray,thick,dashed]
\begin{tikzpicture}[scale=.25]
 \draw (0,0) rectangle (1,4);
 \draw (1,0) rectangle (2,4);
 \draw (2,0) rectangle (4,1);
 \draw (2,1) rectangle (4,4);
 \draw[style=my help lines] (0,4) -- (4,0);
\end{tikzpicture} \\ \updownarrow \\ \tikzstyle{my help lines}=[gray,thick,dashed]
\begin{tikzpicture}[scale=.25]
 \draw (0,0) rectangle (2,3);
 \draw (0,3) rectangle (2,4);
 \draw (2,0) rectangle (3,4);
 \draw (3,0) rectangle (4,4);
 \draw[style=my help lines] (0,4) -- (4,0);
\end{tikzpicture} \\ \end{array}$ & $\begin{array}{c} \begin{array}{cc} 3 & 3\\ \end{array} \\ \updownarrow \\ \begin{array}{cc} 0 & 0\\ \end{array} \\ \end{array}$ \\

\hline

$\begin{array}{c} 1342 \\ \updownarrow \\ 3124 \\ \end{array}$ & $\begin{array}{c} 1342 \\ \updownarrow \\ 3124 \\ \end{array}$ & $\begin{array}{c} \begin{tikzpicture}[scale=.2]
 \draw[gray,very thin] (0,0) grid (5,5);
 \draw[very thick] (2,0) -- (4,0) -- (4,1);
 \draw[very thick] (1,1) -- (3,1) -- (3,2);
 \draw[very thick] (0,2) -- (1,2) -- (1,3) -- (2,3);
 \filldraw (0,2) circle (2pt)
           (1,1) circle (2pt)
           (2,0) circle (2pt)
           (4,1) circle (2pt)
           (3,2) circle (2pt)
           (2,3) circle (2pt);

\end{tikzpicture} \\ \updownarrow \\ \begin{tikzpicture}[scale=.2]
 \draw[gray,very thin] (0,0) grid (5,5);
 \draw[very thick] (2,0) -- (3,0) -- (3,1) -- (4,1);
 \draw[very thick] (1,1) -- (1,2) -- (2,2) -- (3,2);
 \draw[very thick] (0,2) -- (0,3) -- (1,3) -- (2,3);
 \filldraw (0,2) circle (2pt)
           (1,1) circle (2pt)
           (2,0) circle (2pt)
           (4,1) circle (2pt)
           (3,2) circle (2pt)
           (2,3) circle (2pt);

\end{tikzpicture} \\ \end{array}$ &  $\begin{array}{c}{\young(1369,248\eleven,57\ten\twelve)} \\ \updownarrow \\ {\young(1368,259\eleven,47\ten\twelve)} \\ \end{array}$ & $\begin{array}{c} \tikzstyle{my help lines}=[gray,thick,dashed]
\begin{tikzpicture}[scale=.25]
 \draw (0,0) rectangle (1,4);
 \draw (1,0) rectangle (3,2);
 \draw (3,0) rectangle (4,2);
 \draw (1,2) rectangle (4,4);
 \draw[style=my help lines] (0,4) -- (4,0);
\end{tikzpicture} \\ \updownarrow \\ \tikzstyle{my help lines}=[gray,thick,dashed]
\begin{tikzpicture}[scale=.25]
 \draw (0,0) rectangle (3,2);
 \draw (0,2) rectangle (1,4);
 \draw (1,2) rectangle (3,4);
 \draw (3,0) rectangle (4,4);
 \draw[style=my help lines] (0,4) -- (4,0);
\end{tikzpicture} \\ \end{array}$ & $\begin{array}{c} \begin{array}{cc} 3 & 2\\ \end{array} \\ \updownarrow \\ \begin{array}{cc} 1 & 0\\ \end{array} \\ \end{array}$ \\

\hline

$\updownarrow$ & $\updownarrow$ & $\updownarrow$ & $\updownarrow$ & $\updownarrow$ & $\updownarrow$ \\

 $\begin{array}{c} 1423 \\ \updownarrow \\ 2314 \\ \end{array}$ & $\begin{array}{c} 1423 \\ \updownarrow \\ 2314 \\ \end{array}$ & $\begin{array}{c} \begin{tikzpicture}[scale=.2]
 \draw[gray,very thin] (0,0) grid (5,5);
 \draw[very thick] (2,0) -- (4,0) -- (4,1);
 \draw[very thick] (1,1) -- (2,1) -- (2,2) -- (3,2);
 \draw[very thick] (0,2) -- (1,2) -- (1,3) -- (2,3);
 \filldraw (0,2) circle (2pt)
           (1,1) circle (2pt)
           (2,0) circle (2pt)
           (4,1) circle (2pt)
           (3,2) circle (2pt)
           (2,3) circle (2pt);

\end{tikzpicture} \\ \updownarrow \\ \begin{tikzpicture}[scale=.2]
 \draw[gray,very thin] (0,0) grid (5,5);
 \draw[very thick] (2,0) -- (3,0) -- (3,1) -- (4,1);
 \draw[very thick] (1,1) -- (2,1) -- (2,2) -- (3,2);
 \draw[very thick] (0,2) -- (0,3) -- (1,3) -- (2,3);
 \filldraw (0,2) circle (2pt)
           (1,1) circle (2pt)
           (2,0) circle (2pt)
           (4,1) circle (2pt)
           (3,2) circle (2pt)
           (2,3) circle (2pt);

\end{tikzpicture} \\ \end{array}$ &  $\begin{array}{c} {\young(1357,269\eleven,48\ten\twelve)} \\ \updownarrow \\ {\young(1359,247\eleven,68\ten\twelve)} \\ \end{array}$ & $\begin{array}{c} \tikzstyle{my help lines}=[gray,thick,dashed]
\begin{tikzpicture}[scale=.25]
 \draw (0,0) rectangle (1,4);
 \draw (1,0) rectangle (4,1);
 \draw (1,1) rectangle (2,4);
 \draw (2,1) rectangle (4,4);
 \draw[style=my help lines] (0,4) -- (4,0);
\end{tikzpicture} \\ \updownarrow \\ \tikzstyle{my help lines}=[gray,thick,dashed]
\begin{tikzpicture}[scale=.25]
 \draw (0,0) rectangle (2,3);
 \draw (2,0) rectangle (3,3);
 \draw (0,3) rectangle (3,4);
 \draw (3,0) rectangle (4,4);
 \draw[style=my help lines] (0,4) -- (4,0);
\end{tikzpicture} \\ \end{array}$ & $\begin{array}{c} \begin{array}{cc} 3 & 1\\ \end{array} \\ \updownarrow \\ \begin{array}{cc} 2 & 0\\ \end{array} \\ \end{array}$ \\

\hline

 $\begin{array}{c} 2341 \\ \updownarrow \\ 4123 \\ \end{array}$ & $\begin{array}{c} 2341 \\ \updownarrow \\ 4123 \\ \end{array}$ & $\begin{array}{c} \begin{tikzpicture}[scale=.2]
 \draw[gray,very thin] (0,0) grid (5,5);
 \draw[very thick] (2,0) -- (3,0) -- (4,0) -- (4,1);
 \draw[very thick] (1,1) -- (2,1) -- (3,1) -- (3,2);
 \draw[very thick] (0,2) -- (0,3) -- (1,3) -- (2,3);
 \filldraw (0,2) circle (2pt)
           (1,1) circle (2pt)
           (2,0) circle (2pt)
           (4,1) circle (2pt)
           (3,2) circle (2pt)
           (2,3) circle (2pt);

\end{tikzpicture} \\ \updownarrow \\ \begin{tikzpicture}[scale=.2]
 \draw[gray,very thin] (0,0) grid (5,5);
 \draw[very thick] (2,0) -- (3,0) -- (4,0) -- (4,1);
 \draw[very thick] (1,1) -- (1,2) -- (2,2) -- (3,2);
 \draw[very thick] (0,2) -- (0,3) -- (1,3) -- (2,3);
 \filldraw (0,2) circle (2pt)
           (1,1) circle (2pt)
           (2,0) circle (2pt)
           (4,1) circle (2pt)
           (3,2) circle (2pt)
           (2,3) circle (2pt);

\end{tikzpicture} \\ \end{array}$ &  $\begin{array}{c} {\young(1469,258\eleven,37\ten\twelve)} \\ \updownarrow \\ {\young(136\ten,258\eleven,479\twelve)} \\ \end{array}$ & $\begin{array}{c} \tikzstyle{my help lines}=[gray,thick,dashed]
\begin{tikzpicture}[scale=.25]
 \draw (0,0) rectangle (2,3);
 \draw (0,3) rectangle (4,4);
 \draw (2,0) rectangle (3,3);
 \draw (3,0) rectangle (4,3);
 \draw[style=my help lines] (0,4) -- (4,0);
\end{tikzpicture} \\ \updownarrow \\ \tikzstyle{my help lines}=[gray,thick,dashed]
\begin{tikzpicture}[scale=.25]
 \draw (0,0) rectangle (4,1);
 \draw (0,1) rectangle (1,4);
 \draw (1,1) rectangle (2,4);
 \draw (2,1) rectangle (4,4);
 \draw[style=my help lines] (0,4) -- (4,0);
\end{tikzpicture} \\ \end{array}$ & $\begin{array}{c} \begin{array}{cc} 2 & 2\\ \end{array} \\ \updownarrow \\ \begin{array}{cc} 1 & 1\\ \end{array} \\ \end{array}$ \\

\hline

 $\begin{array}{c} 1324 \\ \circlearrowleft \\ \end{array}$ & $\begin{array}{c} 1324 \\ \circlearrowleft \\ \end{array}$ & $\begin{array}{c} \begin{tikzpicture}[scale=.2]
 \draw[gray,very thin] (0,0) grid (5,5);
 \draw[very thick] (2,0) -- (3,0) -- (3,1) -- (4,1);
 \draw[very thick] (1,1) -- (2,1) -- (2,2) -- (3,2);
 \draw[very thick] (0,2) -- (1,2) -- (1,3) -- (2,3);
 \filldraw (0,2) circle (2pt)
           (1,1) circle (2pt)
           (2,0) circle (2pt)
           (4,1) circle (2pt)
           (3,2) circle (2pt)
           (2,3) circle (2pt);

\end{tikzpicture} \\ \circlearrowleft \\ \end{array}$ &  $\begin{array}{c}{\young(1357,249\eleven,68\ten\twelve)} \\ \circlearrowleft \\ \end{array}$ & $\begin{array}{c} \tikzstyle{my help lines}=[gray,thick,dashed]
\begin{tikzpicture}[scale=.25]
 \draw (0,0) rectangle (1,4);
 \draw (1,0) rectangle (3,2);
 \draw (1,2) rectangle (3,4);
 \draw (3,0) rectangle (4,4);
 \draw[style=my help lines] (0,4) -- (4,0);
\end{tikzpicture} \\ \circlearrowleft \\ \end{array}$ & $\begin{array}{c} \begin{array}{cc} 3 & 0\\ \end{array} \\ \circlearrowleft \\ \end{array}$ \\

\hline

  $\begin{array}{c} 3412 \\ \circlearrowleft \\ \end{array}$ & $\begin{array}{c} 3142 \\ \circlearrowleft \\ \end{array}$ & $\begin{array}{c} \begin{tikzpicture}[scale=.2]
 \draw[gray,very thin] (0,0) grid (5,5);
 \draw[very thick] (2,0) -- (3,0) -- (4,0) -- (4,1);
 \draw[very thick] (1,1) -- (2,1) -- (2,2) -- (3,2);
 \draw[very thick] (0,2) -- (0,3) -- (1,3) -- (2,3);
 \filldraw (0,2) circle (2pt)
           (1,1) circle (2pt)
           (2,0) circle (2pt)
           (4,1) circle (2pt)
           (3,2) circle (2pt)
           (2,3) circle (2pt);

\end{tikzpicture} \\ \circlearrowleft \\ \end{array}$ &  $\begin{array}{c} {\young(1379,258\eleven,46\ten\twelve)} \\ \circlearrowleft \\ \end{array}$ & $\begin{array}{c} \tikzstyle{my help lines}=[gray,thick,dashed]
\begin{tikzpicture}[scale=.25]
 \draw (0,0) rectangle (3,2);
 \draw (0,2) rectangle (1,4);
 \draw (1,2) rectangle (4,4);
 \draw (3,0) rectangle (4,2);
 \draw[style=my help lines] (0,4) -- (4,0);
\end{tikzpicture} \\ \circlearrowleft \\ \end{array}$ & $\begin{array}{c} \begin{array}{cc} 2 & 1\\ \end{array} \\ \circlearrowleft \\ \end{array}$ \\

\hline

\end{tabular}

\end{table}

\begin{table}
\caption{Baxter objects of order $(k,\ell)=(1,2)$}
\label{Table 3}
\centering
$\Theta_{k,\ell}(q)=1+q+2q^2+2q^3+2q^4+q^5+q^6$\\
$\Theta_{k,\ell}(1) =10$\\
$\Theta_{k,\ell}(-1)=2$\\
\begin{tabular}[t]{|c|c|c|c|c|c|}

 \hline
Baxter & Twisted Baxter  & Baxter & Baxter  & Diagonal  & Baxter Plane  \\
Permutations &  Permutations  & Paths & Tableaux &  Rectangulations &  Partitions  \\
\hline
\hline

 $\begin{array}{c} 1432 \\ \updownarrow \\ 3214 \\ \end{array}$ & $\begin{array}{c} 1432 \\ \updownarrow \\ 3214 \\ \end{array}$ & $\begin{array}{c} \begin{tikzpicture}[scale=.2]
 \draw[gray,very thin] (0,0) grid (5,5);
 \draw[very thick] (2,0) -- (3,0) -- (3,1) -- (3,2);
 \draw[very thick] (1,1) -- (2,1) -- (2,2) -- (2,3);
 \draw[very thick] (0,2) -- (0,3) -- (0,4) -- (1,4);
 \filldraw (0,2) circle (2pt)
           (1,1) circle (2pt)
           (2,0) circle (2pt)
           (2,3) circle (2pt)
           (3,2) circle (2pt)
           (1,4) circle (2pt);

\end{tikzpicture} \\ \updownarrow \\ \begin{tikzpicture}[scale=.2]
 \draw[gray,very thin] (0,0) grid (5,5);
 \draw[very thick] (2,0) -- (3,0) -- (3,1) -- (3,2);
 \draw[very thick] (1,1) -- (1,2) -- (1,3) -- (2,3);
 \draw[very thick] (0,2) -- (0,3) -- (0,4) -- (1,4);
 \filldraw (0,2) circle (2pt)
           (1,1) circle (2pt)
           (2,0) circle (2pt)
           (2,3) circle (2pt)
           (3,2) circle (2pt)
           (1,4) circle (2pt);

\end{tikzpicture} \\ \end{array}$ &  $\begin{array}{c} {\young(1369,247\eleven,58\ten\twelve)} \\ \updownarrow \\ {\young(1358,269\eleven,47\ten\twelve)} \\ \end{array}$ & $\begin{array}{c} \tikzstyle{my help lines}=[gray,thick,dashed]
\begin{tikzpicture}[scale=.25]
 \draw (0,0) rectangle (1,4);
 \draw (1,0) rectangle (4,1);
 \draw (1,1) rectangle (4,2);
 \draw (1,2) rectangle (4,4);
 \draw[style=my help lines] (0,4) -- (4,0);
\end{tikzpicture} \\ \updownarrow \\ \tikzstyle{my help lines}=[gray,thick,dashed]
\begin{tikzpicture}[scale=.25]
 \draw (0,0) rectangle (3,2);
 \draw (0,2) rectangle (3,3);
 \draw (0,3) rectangle (3,4);
 \draw (3,0) rectangle (4,4);
 \draw[style=my help lines] (0,4) -- (4,0);
\end{tikzpicture} \\ \end{array}$ & $\begin{array}{c} 2 \\ 2\\ \updownarrow \\ 1 \\ 1 \\ \end{array}$ \\

\hline

 $\begin{array}{c} 2431 \\ \updownarrow \\ 4213 \\ \end{array}$ & $\begin{array}{c} 2431 \\ \updownarrow \\ 4213 \\ \end{array}$ & $\begin{array}{c} \begin{tikzpicture}[scale=.2]
 \draw[gray,very thin] (0,0) grid (5,5);
 \draw[very thick] (2,0) -- (3,0) -- (3,1) -- (3,2);
 \draw[very thick] (1,1) -- (2,1) -- (2,2) -- (2,3);
 \draw[very thick] (0,2) -- (0,3) -- (1,3) -- (1,4);
 \filldraw (0,2) circle (2pt)
           (1,1) circle (2pt)
           (2,0) circle (2pt)
           (2,3) circle (2pt)
           (3,2) circle (2pt)
           (1,4) circle (2pt);

\end{tikzpicture} \\ \updownarrow \\ \begin{tikzpicture}[scale=.2]
 \draw[gray,very thin] (0,0) grid (5,5);
 \draw[very thick] (2,0) -- (2,1) -- (3,1) -- (3,2);
 \draw[very thick] (1,1) -- (1,2) -- (1,3) -- (2,3);
 \draw[very thick] (0,2) -- (0,3) -- (0,4) -- (1,4);
 \filldraw (0,2) circle (2pt)
           (1,1) circle (2pt)
           (2,0) circle (2pt)
           (2,3) circle (2pt)
           (3,2) circle (2pt)
           (1,4) circle (2pt);

\end{tikzpicture} \\ \end{array}$ &  $\begin{array}{c} {\young(1469,257\eleven,38\ten\twelve)} \\ \updownarrow \\ {\young(135\ten,268\eleven,479\twelve)} \\ \end{array}$ & $\begin{array}{c} \tikzstyle{my help lines}=[gray,thick,dashed]
\begin{tikzpicture}[scale=.25]
 \draw (0,0) rectangle (2,3);
 \draw (0,3) rectangle (4,4);
 \draw (2,0) rectangle (4,1);
 \draw (2,1) rectangle (4,3);
 \draw[style=my help lines] (0,4) -- (4,0);
\end{tikzpicture} \\ \updownarrow \\ \tikzstyle{my help lines}=[gray,thick,dashed]
\begin{tikzpicture}[scale=.25]
 \draw (0,0) rectangle (4,1);
 \draw (0,1) rectangle (2,3);
 \draw (0,3) rectangle (2,4);
 \draw (2,1) rectangle (4,4);
 \draw[style=my help lines] (0,4) -- (4,0);
\end{tikzpicture} \\ \end{array}$ & $\begin{array}{c} 3 \\ 2\\ \updownarrow \\ 1 \\ 0 \\ \end{array}$ \\

\hline

 $\begin{array}{c} 3241 \\ \updownarrow \\ 4132 \\ \end{array}$ & $\begin{array}{c} 3241 \\ \updownarrow \\ 4132 \\ \end{array}$ & $\begin{array}{c} \begin{tikzpicture}[scale=.2]
 \draw[gray,very thin] (0,0) grid (5,5);
 \draw[very thick] (2,0) -- (3,0) -- (3,1) -- (3,2);
 \draw[very thick] (1,1) -- (1,2) -- (2,2) -- (2,3);
 \draw[very thick] (0,2) -- (0,3) -- (1,3) -- (1,4);
 \filldraw (0,2) circle (2pt)
           (1,1) circle (2pt)
           (2,0) circle (2pt)
           (2,3) circle (2pt)
           (3,2) circle (2pt)
           (1,4) circle (2pt);

\end{tikzpicture} \\ \updownarrow \\ \begin{tikzpicture}[scale=.2]
 \draw[gray,very thin] (0,0) grid (5,5);
 \draw[very thick] (2,0) -- (2,1) -- (3,1) -- (3,2);
 \draw[very thick] (1,1) -- (1,2) -- (2,2) -- (2,3);
 \draw[very thick] (0,2) -- (0,3) -- (0,4) -- (1,4);
 \filldraw (0,2) circle (2pt)
           (1,1) circle (2pt)
           (2,0) circle (2pt)
           (2,3) circle (2pt)
           (3,2) circle (2pt)
           (1,4) circle (2pt);

\end{tikzpicture} \\ \end{array}$ &  $\begin{array}{c} {\young(1468,259\eleven,37\ten\twelve)} \\ \updownarrow \\ {\young(136\ten,248\eleven,579\twelve)} \\ \end{array}$ & $\begin{array}{c} \tikzstyle{my help lines}=[gray,thick,dashed]
\begin{tikzpicture}[scale=.25]
 \draw (0,0) rectangle (3,2);
 \draw (0,2) rectangle (3,3);
 \draw (0,3) rectangle (4,4);
 \draw (3,0) rectangle (4,3);
 \draw[style=my help lines] (0,4) -- (4,0);
\end{tikzpicture} \\ \updownarrow \\ \tikzstyle{my help lines}=[gray,thick,dashed]
\begin{tikzpicture}[scale=.25]
 \draw (0,0) rectangle (4,1);
 \draw (0,1) rectangle (1,4);
 \draw (1,1) rectangle (4,2);
 \draw (1,2) rectangle (4,4);
 \draw[style=my help lines] (0,4) -- (4,0);
\end{tikzpicture} \\ \end{array}$ & $\begin{array}{c} 3 \\ 1\\ \updownarrow \\ 2 \\ 0 \\ \end{array}$ \\

\hline

 $\begin{array}{c} 3421 \\ \updownarrow \\ 4312 \\ \end{array}$ & $\begin{array}{c} 3421 \\ \updownarrow \\ 4312 \\ \end{array}$ & $\begin{array}{c} \begin{tikzpicture}[scale=.2]
 \draw[gray,very thin] (0,0) grid (5,5);
 \draw[very thick] (2,0) -- (3,0) -- (3,1) -- (3,2);
 \draw[very thick] (1,1) -- (2,1) -- (2,2) -- (2,3);
 \draw[very thick] (0,2) -- (1,2) -- (1,3) -- (1,4);
 \filldraw (0,2) circle (2pt)
           (1,1) circle (2pt)
           (2,0) circle (2pt)
           (2,3) circle (2pt)
           (3,2) circle (2pt)
           (1,4) circle (2pt);

\end{tikzpicture} \\ \updownarrow \\ \begin{tikzpicture}[scale=.2]
 \draw[gray,very thin] (0,0) grid (5,5);
 \draw[very thick] (2,0) -- (2,1) -- (2,2) -- (3,2);
 \draw[very thick] (1,1) -- (1,2) -- (1,3) -- (2,3);
 \draw[very thick] (0,2) -- (0,3) -- (0,4) -- (1,4);
 \filldraw (0,2) circle (2pt)
           (1,1) circle (2pt)
           (2,0) circle (2pt)
           (2,3) circle (2pt)
           (3,2) circle (2pt)
           (1,4) circle (2pt);

\end{tikzpicture} \\ \end{array}$ &  $\begin{array}{c} {\young(1479,258\eleven,36\ten\twelve)} \\ \updownarrow \\ {\young(137\ten,258\eleven,469\twelve)} \\ \end{array}$ & $\begin{array}{c} \tikzstyle{my help lines}=[gray,thick,dashed]
\begin{tikzpicture}[scale=.25]
 \draw (0,0) rectangle (3,2);
 \draw (0,2) rectangle (4,3);
 \draw (0,3) rectangle (4,4);
 \draw (3,0) rectangle (4,2);
 \draw[style=my help lines] (0,4) -- (4,0);
\end{tikzpicture} \\ \updownarrow \\ \tikzstyle{my help lines}=[gray,thick,dashed]
\begin{tikzpicture}[scale=.25]
 \draw (0,0) rectangle (4,1);
 \draw (0,1) rectangle (4,2);
 \draw (0,2) rectangle (1,4);
 \draw (1,2) rectangle (4,4);
 \draw[style=my help lines] (0,4) -- (4,0);
\end{tikzpicture} \\ \end{array}$ & $\begin{array}{c} 3 \\ 3\\ \updownarrow \\ 0 \\ 0 \\ \end{array}$ \\

\hline

 $\begin{array}{c} 2143 \\ \circlearrowleft \\ \end{array}$ & $\begin{array}{c} 2143 \\ \circlearrowleft \\ \end{array}$ & $\begin{array}{c} \begin{tikzpicture}[scale=.2]
 \draw[gray,very thin] (0,0) grid (5,5);
 \draw[very thick] (2,0) -- (3,0) -- (3,1) -- (3,2);
 \draw[very thick] (1,1) -- (1,2) -- (2,2) -- (2,3);
 \draw[very thick] (0,2) -- (0,3) -- (0,4) -- (1,4);
 \filldraw (0,2) circle (2pt)
           (1,1) circle (2pt)
           (2,0) circle (2pt)
           (2,3) circle (2pt)
           (3,2) circle (2pt)
           (1,4) circle (2pt);

\end{tikzpicture} \\ \circlearrowleft \\ \end{array}$ &  $\begin{array}{c} {\young(1368,249\eleven,57\ten\twelve)} \\ \circlearrowleft \\ \end{array}$ & $\begin{array}{c} \tikzstyle{my help lines}=[gray,thick,dashed]
\begin{tikzpicture}[scale=.25]
 \draw (0,0) rectangle (2,3);
 \draw (0,3) rectangle (2,4);
 \draw (2,0) rectangle (4,1);
 \draw (2,1) rectangle (4,4);
 \draw[style=my help lines] (0,4) -- (4,0);
\end{tikzpicture} \\ \circlearrowleft \\ \end{array}$ & $\begin{array}{c} 2 \\ 1\\ \circlearrowleft \\ \end{array}$ \\

\hline

 $\begin{array}{c} 4231 \\ \circlearrowleft \\ \end{array}$ & $\begin{array}{c} 4231 \\ \circlearrowleft \\ \end{array}$ & $\begin{array}{c} \begin{tikzpicture}[scale=.2]
 \draw[gray,very thin] (0,0) grid (5,5);
 \draw[very thick] (2,0) -- (2,1) -- (3,1) -- (3,2);
 \draw[very thick] (1,1) -- (1,2) -- (2,2) -- (2,3);
 \draw[very thick] (0,2) -- (0,3) -- (1,3) -- (1,4);
 \filldraw (0,2) circle (2pt)
           (1,1) circle (2pt)
           (2,0) circle (2pt)
           (2,3) circle (2pt)
           (3,2) circle (2pt)
           (1,4) circle (2pt);

\end{tikzpicture} \\ \circlearrowleft \\ \end{array}$ &  $\begin{array}{c} {\young(146\ten,258\eleven,379\twelve)} \\ \circlearrowleft \\ \end{array}$ & $\begin{array}{c} \tikzstyle{my help lines}=[gray,thick,dashed]
\begin{tikzpicture}[scale=.25]
 \draw (0,0) rectangle (4,1);
 \draw (0,1) rectangle (2,3);
 \draw (2,1) rectangle (4,3);
 \draw (0,3) rectangle (4,4);
 \draw[style=my help lines] (0,4) -- (4,0);
\end{tikzpicture} \\ \circlearrowleft \\ \end{array}$ & $\begin{array}{c} 3 \\ 0\\ \circlearrowleft \\ \end{array}$ \\

\hline

\end{tabular}

\end{table}

\end{document}